\documentclass[12pt]{article}
\usepackage{amsmath,amssymb}
\usepackage{amscd}
\usepackage[all]{xy}
\title{Anabelian geometry and descent obstructions on moduli spaces
\thanks{
Keywords: Anabelian geometry, moduli spaces, abelian varieties,
descent obstruction.
Mathematical subject classification: Primary 11G35; Secondary:14G05,14G35}
}

\author{Stefan Patrikis, Jos\'e Felipe Voloch, Yuri G. Zarhin}
\renewcommand\footnotemark{}

\DeclareFontFamily{U}{wncy}{}
\DeclareFontShape{U}{wncy}{m}{n}{%
   <5>wncyr5%
   <6>wncyr6%
   <7>wncyr7%
   <8>wncyr8%
   <9>wncyr9%
   <10>wncyr10%
   <11>wncyr10%
   <12>wncyr6%
   <14>wncyr7%
   <17>wncyr8%
   <20>wncyr10%
   <25>wncyr10}{}
\DeclareMathAlphabet{\cyrille}{U}{wncy}{m}{n}

\def \Pic{{\rm Pic}}
\def \Out{{\rm Out}}

\def\Oc{{\mathcal O}}
\def\fchar{\mathrm{char}}
\newcommand{\gal}[1]{G_{#1}} 
\newcommand{\into}{\hookrightarrow}
\newcommand{\onto}{\twoheadrightarrow}
\newcommand{\mc}{\mathcal}

\newcommand{\mr}{\mathrm}

\DeclareMathOperator{\Rep}{Rep}
\DeclareMathOperator{\End}{End}
\DeclareMathOperator{\Fil}{Fil}
\DeclareMathOperator{\Res}{Res}
\DeclareMathOperator{\Ind}{Ind}
\DeclareMathOperator{\gr}{gr}
\DeclareMathOperator{\Aut}{Aut}

\DeclareMathOperator{\tr}{tr}
\newcommand{\Q}{\mathbb{Q}}
\newcommand{\Qb}{\overline{\mathbb{Q}}}
\newcommand{\Z}{\mathbb{Z}}
\newcommand{\CC}{\mathbb{C}}

                                                     \newcommand{\MM}{\mathbb{M}}

\newcommand{\Ql}{\mathbb{Q}_{\ell}}
\newcommand{\Qlb}{\overline{\mathbb{Q}}_\ell}

\newtheorem{thm}{Theorem}[section]
\newtheorem{prop}[thm]{Proposition}
\newtheorem{lem}[thm]{Lemma}
\newtheorem{lemma}[thm]{Lemma}

\newtheorem{cor}[thm]{Corollary}

          \def \calK {{\mathcal K}}
          \def \calL {{\mathcal L}}
          \def \calA {{\mathcal Y}}

\def \A {{\cal A}}
\def \M {{\cal M}}

\def \Cu {{\cal C}}

\def \Spec {{\rm{Spec}}}

\def \Gal {{\rm{Gal}}}
\def \dim {{\rm{dim\,}}}

\def \Hom {{\rm {Hom}}}

\def\smallsquare{\vbox{\hrule\hbox{\vrule height 1 ex\kern 1 ex\vrule}\hrule}}

\def\endproof{\hfill \smallsquare\vskip 3mm}

\def \abstract{\paragraph{Abstract. }}
\def \proof{\paragraph{Proof. }}

\begin{document}

\maketitle

\begin{abstract}
We study the section conjecture of anabelian geometry and the sufficiency of the finite descent obstruction to the Hasse principle for the moduli spaces of principally polarized abelian varieties and of curves over number fields. For the former we show that the section conjecture fails and the finite descent obstruction holds for a general class of adelic points, assuming several well-known conjectures. This is done by relating the problem to a local-global principle for Galois representations. For the latter, we show how the sufficiency of the finite descent obstruction implies the same for all hyperbolic curves.
\end{abstract}

\section{Introduction}

Anabelian geometry is a program proposed by Grothendieck (\cite{G1,G2}) which suggests that for a certain class of
varieties (called anabelian but, as yet, undefined) over a number field, one can recover the varieties from their \'etale
fundamental group together with the Galois action of the absolute Galois group of the number field. Precise conjectures
exist only for curves and some of them have been proved, notably by Mochizuki (\cite{M}). Grothendieck suggested that
moduli spaces of curves and abelian varieties (the latter perhaps less emphatically) should be anabelian. Already Ihara
and Nakamura \cite{IN} have shown that moduli spaces of abelian varieties should not be anabelian as one cannot
recover their automorphism group from the fundamental group and we will further show that other anabelian properties
fail in this case. 

The finite descent obstruction is a construction that describes a subset of the adelic points of a variety over a number
field containing the closure of the rational (or integral) points and is conjectured, for hyperbolic curves (Stoll, 
\cite{Stoll} in the projective case and Harari and Voloch \cite{HV} in the affine case) to equal that closure. 
It's not unreasonable to conjecture the
same for all anabelian varieties. The relationship between the
finite descent obstruction and the section
conjecture in anabelian geometry has been discussed by
Harari and Stix \cite{HS}, Stix \cite{Stix}, Section 11 and others. We will review the
relevant definitions below, although our point of view will be slightly different.

The purpose of this paper is to study the section conjecture of anabelian geometry and the finite descent obstruction
for the moduli spaces of principally polarized abelian varieties and of curves over number fields. For
the moduli of abelian varieties we show
that the section conjecture fails in general and
that both the section conjecture and finite descent obstruction hold for a general class of adelic points,
assuming many established conjectures in arithmetic geometry  (specifically, we assume the 
Hodge, Tate, Fontaine-Mazur and Grothendieck-Serre conjectures, in the precise forms stated in section \ref{abvar}). 
This is done by converting the question into one about Galois representations. 

The section conjecture predicts that sections of the fundamental exact sequence (Section \ref{abvar}, eq. \ref{fund})
 of an anabelian variety
over a number field correspond to rational points. In this paper, we look at the sections of the 
fundamental exact sequence of the moduli spaces of principally polarized abelian varieties that, locally
at every place of the ground field come from a point rational over the completion, which moreover is integral
for all but finitely many places. This set is denoted  $S_0(K,\A_g)$ and defined precisely
at the end of Section \ref{prelims}. We explain, in section \ref{abvar}, how
sections of the 
fundamental exact sequence of the moduli spaces of principally polarized abelian varieties correspond to
Galois representations and prove, Theorem \ref{coro}, the following result.

\begin{thm}
Assume the Hodge, Tate, Fontaine-Mazur, and Grothendieck-Serre conjectures. Let $K$ be a number field.
Suppose $s \in S_0(K,\A_g)$ gives rise to a system of $\ell$-adic Galois representations one of which is absolutely irreducible. Then there exists, up to isomorphism, a unique principally polarized abelian variety which, viewed as
point of $\A_g(K)$, induces (up to conjugation) the section $s$.
\end{thm}

We also give examples 
(see Theorems \ref{counterexample} and \ref{counterexample2}) showing that weaker versions of the above result do not hold. Specifically,
the local conditions cannot be weakened to hold almost everywhere, for instance.

For the moduli of curves, we show how combining some of our results and assuming
sufficiency of finite descent obstruction for the moduli of curves, we deduce
the sufficiency of finite descent obstruction for all hyperbolic curves.

In the next section we give more precise definitions of the objects we use and in the following two sections we
give the applications mentioned above.

\section{Preliminaries}
\label{prelims}

Let $X/K$ be a smooth geometrically connected variety over a field $K$.
Let $G_K$ be the absolute Galois group of $K$ and $\bar{X}$ the base-change of $X$ to an algebraic closure
of $K$. We denote by $\pi_1(.)$ the algebraic fundamental group functor
on (geometrically pointed) schemes and we omit base-points from the
notation. We have the fundamental exact sequence
\begin{equation}
\label{fund}
1 \rightarrow \pi_1(\bar{X}) \rightarrow \pi_1(X) \rightarrow G_K \rightarrow 1.
\end{equation}
The map $p_X: \pi_1(X) \rightarrow G_K$ from the above sequence is obtained by functoriality from the
structural morphism $X \to \Spec K$. Grothendieck's anabelian program is to specify a class of varieties,
termed anabelian, for which the varieties and morphisms between them can be recovered from the corresponding
fundamental groups together with the corresponding maps $p_X$ when the ground field is finitely generated over
$\mathbb{Q}$. As this is very vague, we single out here two special cases with precise statements. The first is a
(special case of a) theorem of Mochizuki \cite{M}
which implies part of Grothendieck's conjectures for curves but also
extends it by considering $p$-adic fields.

\begin{thm}
\label{moch}
(Mochizuki) Let $X,Y$ be smooth projective curves of genus bigger than one
over a field $K$ which is a subfield of a finitely generated extension of $\Q_p$. If
there is an isomorphism from $\pi_1(X)$ to $\pi_1(Y)$ inducing the identity on $G_K$ via $p_X,p_Y$, then
$X$ is isomorphic to $Y$.
\end{thm}

A point $P \in X(K)$ gives, by functoriality, a section
$G_K \rightarrow \pi_1(X)$ of the fundamental exact sequence
(\ref{fund}) well-defined up to conjugation by an element of $\pi_1(\bar{X})$
(the indeterminacy is because of base points).

We denote by $H(K,X)$ the set of sections $G_K \rightarrow \pi_1(X)$ modulo conjugation by $\pi_1(\bar{X})$ and
we denote by $\sigma_{X/K}: X(K) \to H(K,X)$ the map that associates
to a point the class of its corresponding section, as above,
and we call it the section map.
As part of the anabelian program, it is expected that
$\sigma_{X/K}$ is a bijection if $X$ is projective, anabelian
and $K$ is finitely generated over its prime field. This is widely
believed in the case of hyperbolic curves over number fields and
is usually referred as the section conjecture. 
For a similar statement
in the non-projective case, one needs to consider the so-called cuspidal
sections, 
see \cite{Stix}, Section 18. 
Although we will discuss non-projective varieties
in what follows, we will not need to specify the notion of cuspidal sections.
The reason for this is that we will be considering sections that locally come
from points (the Selmer set defined below) and these will not be cuspidal.

We remark that the choice of a particular section
$s_0: G_K \rightarrow \pi_1(X)$ induces an action of $G_K$ on
$\pi_1(\bar{X}), x \mapsto s_0(\gamma)xs_0(\gamma)^{-1}$.
For an arbitrary section $s:G_K \rightarrow \pi_1(X)$
the map $\gamma \mapsto s(\gamma)s_0(\gamma)^{-1}$
is a $1$-cocycle for the above action of
$G_K$ on $\pi_1(\bar{X})$ and this induces a bijection
$H^1(G_K,\pi_1(\bar{X})) \to H(K,X)$. We stress that this only holds
when $H(K,X)$ is non-empty and a choice of $s_0$ can be made.
It is possible for $H(K,X)$ to be empty, in which case there is
no natural choice of action of $G_K$ on $\pi_1(\bar{X})$
in order to define $H^1(G_K,\pi_1(\bar{X}))$ which would be non-empty
in any case, if defined.

Let $X/K$ be as above, where $K$ is now a number field.
If $v$ is a place of $K$, we have the completion $K_v$ and
a fixed inclusion $\overline{K} \subset \overline{K}_v$ induces a map $\alpha_v: G_{K_v} \to G_K$
and a map $\beta_v: \pi_1(X_v) \to \pi_1(X)$,
where $X_v$ is the base-change of $X$ to $K_v$.
We define the Selmer set of $X/K$ as the set $S(K,X) \subset H(K,X)$ consisting of the equivalence
classes of sections $s$ such that for all places $v$, there exists $P_v \in X(K_v)$ with
$s \circ \alpha_v = \beta_v \circ \sigma_{X_v/K_v}(P_v)$.
Note that if $v$ is complex, then the condition at $v$ is vacuous and that if $v$ is real,
$\sigma_{X_v/K_v}$ factors through $X(K_v)_\bullet$,
the set of connected components of $X(K_v)$,
equipped with the quotient topology (see \cite{M2,Pal}). In the non-archimedian case,
$X(K_v)$ is totally disconnected so $X(K_v)=X(K_v)_\bullet$ and we have the following diagram:

\[
\xymatrix{
& X(K) \ar[r] \ar[d]_{\sigma_{X/K}} & \prod X(K_v)_\bullet \ar[d]^{\prod \sigma_{X_v/K_v}}& \supset X^f \\
S(K,X) \subset & H(K,X) \ar[r]^{\alpha} & \prod H(K_v,X_v) . \\
}
\]

We define the set $X^f$ (the finite descent obstruction) as the set of points $(P_v)_v \in \prod_v X(K_v)_\bullet$
for which there exists $s \in H(K,X)$ (which is then necessarily an element of $S(K,X)$) satisfying
$s \circ \alpha_v = \beta_v \circ \sigma_{X_v/K_v}(P_v)$ for all places $v$.
Also, it is clear that the image of $X(K)$ is contained in $X^f$.  At least when $X$ is proper, 
$X^f$ is closed (this follows from the compactness of $H(K,X)$, \cite{Stix}, Cor. 45). In that case, one
may consider whether the closure of the image of $X(K)$ in $\prod X(K_v)_\bullet$
equals $X^f$. A related statement is the equality $\sigma_{X/K} (X(K)) = S(K,X)$,
which is implied by the ``section conjecture'', i.e., the bijectivity of $\sigma_{X/K}:X(K) \to H(K,X)$. 
As a specific instance of this relation, we record the following easy fact:

\begin{prop}
\label{descobs}
We have that $X^f = \emptyset$ if and only if $S(K,X) = \emptyset$. 
\end{prop}

\begin{proof}
If $X^f \ne \emptyset$ and $(P_v) \in X^f$, then there exists
$s \in S(K,X)$ with $s \circ \alpha_v = \beta_v \circ \sigma_{X_v/K_v}(P_v)$
for all places $v$, so $S(K,X) \ne \emptyset$.

If $s \in S(K,X)$, there exists
$(P_v)$ with $s \circ \alpha_v = \beta_v \circ \sigma_{X_v/K_v}(P_v)$
for all places $v$. So $(P_v) \in X^f$.
\end{proof}

If $X$ is not projective, then one has to take into account questions of integrality.
We choose an integral model ${\cal{X}}/{\cal{O}}_{S,K}$, where $S$ is a finite set of
places of $K$ and ${\cal{O}}_{S,K}$ is the ring of $S$-integers of $K$. The image of
$X(K)$ in $X^f$ actually lands in the adelic points which are the points that
satisfy $P_v \in {\cal{X}}({\cal{O}}_v)$ for all but finitely many $v$, where ${\cal{O}}_v$
is the local ring at $v$. Similarly, the image of $\sigma_{X/K}$ belongs to the subset
of $S(K,X)$ where the corresponding local points $P_v$ also belong to
${\cal{X}}({\cal{O}}_v)$ for all but finitely many $v$. We denote this subset of $S(K,X)$
by $S_0(K,X)$ and call it the integral Selmer set. We note that $S_0(K,X)$ is
independent of the choice of the model ${\cal{X}}$.

We recall here some basic notions about the Tate module of abelian varieties which will
be used in the next two sections, in order to set notation. 
If $A$ is an abelian variety over the field $K$ then we write $\End(A)$ for its ring of all $K$-endomorphisms and $\End^0(A)$ for the corresponding (finite-dimensional semisimple) $\Q$-algebra $\End(A)\otimes\Q$. 
If $n \ge 3$ is an integer that is not divisible by $\fchar(K)$ and all points of order $n$ on $A$ are defined
over $K$ then, by a theorem of Silverberg \cite{silverberg}, all $\bar{K}$-endomorphisms of $A$ are  
defined over $K$, i.e., lie in $\End(A)$.

If $\ell$ is a prime different from $\fchar(K)$ then we write $T_{\ell}(A)$ for the 
$\Z_{\ell}$-Tate module of $A$ which is a free $\Z_{\ell}$-module of rank $2\dim(A)$ provided with the natural continuous homomorphism
$$\rho_{\ell,A}: G_K \to \Aut_{\Z_{\ell}}(T_{\ell}(A))$$
and the $\Z_{\ell}$-ring embedding
$$e_l:\End(A)\otimes \Z_{\ell} \hookrightarrow \End_{\Z_{\ell}}(T_{\ell}(A)).$$
The image of $\End(A)\otimes \Z_{\ell}$ commutes with $\rho_{\ell,A}(G_K)$. Tensoring by $\Q_{\ell}$ (over $\Z_{\ell}$), we obtain the $\Q_{\ell}$-Tate module of $A$
$$V_{\ell}(A)=T_{\ell}(A)\otimes_{\Z_{\ell}}\Q_{\ell},$$
which is a $2\dim(A)$-dimensional $\Q_{\ell}$-vector space containing
$$T_{\ell}(A)=T_{\ell}(A)\otimes 1$$ as a $\Z_{\ell}$-lattice. We may view $\rho_{\ell,A}$ as an $\ell$-adic representation
$$\rho_{\ell,A}: G_K\to \Aut_{\Z_{\ell}}(T_{\ell}(A))\subset \Aut_{\Q_{\ell}}(V_{\ell}(A))$$
and extend $e_{\ell}$ by $\Q_{\ell}$-linearity to the embedding of $\Q_{\ell}$-algebras
$$\End^0(A)\otimes_{\Q}\Q_{\ell}=\End(A)\otimes \Q_{\ell} \hookrightarrow \End_{\Q_{\ell}}(V_{\ell}(A)),$$
which we still denote by $e_{\ell}$. Further we will identify $\End^0(A)\otimes_{\Q}\Q_{\ell}$ with its image in $\End_{\Q_{\ell}}(V_{\ell}(A))$.

This provides $V_{\ell}(A)$ with the natural structure of $G_K$-module; in addition, 
$\End^0(A)\otimes_{\Q}\Q_{\ell}$ is a $\Q_{\ell}$-(sub)algebra of endomorphisms of the Galois module $V_{\ell}(A)$. In other words,
$$\End^0(A)\otimes_{\Q}\Q_{\ell}\subset
\End_{G_K}(V_{\ell}(A)).$$

\section{Moduli of abelian varieties}
\label{abvar}

The moduli space of principally polarized abelian varieties of dimension $g$ is denoted by $\A_g$. It is actually
a Deligne-Mumford stack or orbifold and we will consider its fundamental group as such.
For a general definition of fundamental groups of stacks including a proof
of the fundamental exact sequence in this generality, see \cite{Zoo}.
For a discussion of the case of $\A_g$, see \cite{Hain}. We can also get
what we need from \cite{IN} (see below) or by working with a level structure
which bring us back to the case of smooth varieties.

As $\A_g$ is defined
over $\Q$, we can consider it over an arbitrary number field $K$.  As per our earlier
conventions, $\bar{\A}_g$ is the base change of $\A_g$ to an algebraic closure of $\Q$ and not a compactification.
In fact, we will not consider a compactification at all here. The topological fundamental group of $\bar{\A}_g$ is
the symplectic group $Sp_{2g}(\Z)$ and the algebraic fundamental group is its profinite completion.
When $g>1$ (which we henceforth assume) $Sp_{2g}(\Z)$ has
the congruence subgroup property (\cite{BLS},\cite{Me}) and therefore
its profinite completion is $Sp_{2g}(\hat{\Z})$.

The group $\pi_1(\A_g)$ is essentially described by the exact sequences (3.2) and (3.3) of
\cite{IN} and it follows that
the set $H(K,\A_g)$ consists of $\hat{\Z}$ representations of $G_K$ of rank $2g$
preserving the symplectic form up to a multiplier given by the cyclotomic character.
Indeed, it is clear that every section gives such a representation and
the converse follows formally from the diagram below, which is a consequence of (3.2) and (3.3) of
\cite{IN}.

In the following we denote the cyclotomic character by $\chi:G_K \to \hat{\Z}^*$.
\[
\xymatrix{
&1 \ar[r] & \pi_1(\bar{\A}_g) \ar[r] \ar[d]^{\cong} & \pi_1({\A}_g) \ar[r] \ar[d] & G_K \ar[r] \ar[d]_{\chi} &1 \\
&1 \ar[r] & Sp_{2g}(\hat{\Z}) \ar[r] & GSp_{2g}(\hat{\Z}) \ar[r] & \hat{\Z}^* \ar[r] &1. \\
}
\]

The coverings of $\bar{\A}_g$ corresponding to the congruence
subgroups of $Sp_{2g}(\hat{\Z})$ are those obtained by adding level
structures. In particular, for an abelian variety $A$,
$\sigma_{\A_g/K}(A) = \prod T_{\ell}(A)$, the product of its Tate
modules considered, as usual, as a $G_K$-module. If $K$ is a number
field, whenever two abelian varieties are mapped to the same point
by $\sigma_{\A_g/K}$, then they are isogenous, by Faltings
(\cite{Faltings}).
The finiteness of  isogeny classes of polarized abelian varieties
over $K$ \cite{Faltings} (see also \cite{Zar_Inv85}) implies that
for any given $K$ and $g$ every fiber of $\sigma_{\A_g/K}$ is
finite.
On the other hand,  $\sigma_{\A_g/K}$ is not necessarily injective to $S_0(K,\A_g)$,  
 see Sect. \ref{twists}. For example,  for
each $g$ there exists $K$ with non-injective $\sigma_{\A_g/K}$.
Regarding surjectivity, we will prove that those elements of $S_0(K,\A_g)$
for which the corresponding Galois representation is absolutely irreducible (see below for the
precise hypothesis and theorem \ref{coro} for a precise statement) are in the image of $\sigma_{\A_g/K}$, assuming
the Fontaine-Mazur conjecture, the Grothendieck-Serre conjecture on
semi-simplicity of $\ell$-adic cohomology of smooth projective varieties, and the Tate and Hodge conjectures.
The integral Selmer set $S_0(K,\A_g)$, defined in the previous section, corresponds
to the set of Galois
representations that are almost everywhere unramified and, locally, come from
abelian varieties (which thus are of good reduction for almost all places of $K$)
and we will also
consider a few variants of the question of surjectivity of $\sigma_{\A_g/K}$ to $S_0(K,\A_g)$
by different local hypotheses and discuss what we can and cannot prove.
A version of this kind of question has also been considered by B. Mazur \cite{Mz}.

Here is the setting. Let $K$ be a number field, with $\gal{K}= \Gal(\overline{K}/K)$.
Fix a finite set of rational primes $S$, and consider a collection of continuous $\ell$-adic representations
\[
\{\rho_{\ell} \colon \gal{K} \to \mr{GL}_N(\Ql) \}_{\ell \not \in S}.
\]
We will say that the collection $\{\rho_{\ell}\}_{\ell \not \in S}$ is \textit{weakly compatible} if there exists a finite set of places $\Sigma$ of $K$ such that
\begin{enumerate}
\item[I.] for all $\ell \not \in S$, $\rho_{\ell}$ is unramified outside the union of $\Sigma$ and the places $\Sigma_{\ell}$ of $K$ dividing $\ell$; and
\item[II.] for all $v \not \in \Sigma \cup \Sigma_{\ell}$, denoting by $fr_v$ a (geometric) frobenius element at $v$, the characteristic polynomial of $\rho_{\ell}(fr_v)$ has rational coefficients and is independent of $\ell \not \in S$. \footnote{These systems were introduced  by Serre \cite{Serre}, who called them {\sl strictly compatible}.}
\end{enumerate}
We will assume $\{\rho_{\ell}\}_{\ell \not \in S}$ is weakly compatible in this sense and moreover satisfies the following three conditions:
\begin{enumerate}
\item For some prime $\ell_0 \not \in S$, $\rho_{\ell_0}$ is de Rham at all places of $K$ above $\ell_0$.
\item For some prime $\ell_1 \not \in S$, $\rho_{\ell_1}$ is absolutely irreducible.
\item For some prime $\ell_2 \not \in S$, and at least one place $v \vert \ell_2$ of $K$, $\rho_{\ell_2}|_{\gal{K_v}}$ is de Rham with Hodge-Tate weights $-1,0$, each with multiplicity $\frac{N}{2}$. (Note that this condition holds if there exists an abelian variety $A_v/K_v$ such that $\rho_{\ell_2}|_{\gal{K_v}} \cong V_{\ell_2}(A_v)$.
\end{enumerate}
We note that the $\ell_i$ need not be distinct, and in fact in the anabelian application (Theorem \ref{coro}) they can all be taken to be equal; we have phrased things this way to have hypotheses that most clearly match the form of the argument. 

Our aim is to prove the following:
\begin{thm}\label{main}
Assume the Hodge, Tate, Fontaine-Mazur, and Grothendieck-Serre conjectures, and suppose that the set $S$ is empty. Then there exists an abelian variety $A$ over $K$ such that $\rho_{\ell} \cong V_{\ell}(A)$ for all $\ell$.
\end{thm}
We begin by making precise the combined implications of the Grothendieck-Serre, Tate, and Fontaine-Mazur conjectures (the Hodge conjecture will only be used later, in the proof of Lemma \ref{hodge}). For any field $k$ and characteristic zero field $E$, let $\mc{M}_{k, E}$ denote the category of pure homological motives over $k$ with coefficients in $E$ (omitting $E$ from the notation will mean $E=\Q$). 
\begin{lemma}\label{motivesrecall}
Assume the Tate conjecture for all finitely-generated extensions $k$ of $\Q$. Then:
\begin{enumerate}
\item The Lefschetz Standard Conjecture holds for all fields of characteristic zero.
\item All of the Standard Conjectures (namely, the K\"{u}nneth and Hodge Standard Conjectures, and the agreement of numerical and homological equivalence) hold for all fields of characteristic zero.
\item For any field $k$ that can be embedded in $\CC$, the category $\mc{M}_{k}$ is a semi-simple neutral Tannakian category over $\Q$. 
\item For any finitely-generated $k/\Q$, the \'{e}tale $\ell$-adic realization functor
\[
\mc{M}_{k, \Q_{\ell}} \to \mathrm{Rep}_{\Q_{\ell}}(G_k),
\]
valued in the category of continuous $\ell$-adic representations of $G_k$, is fully faithful.
\end{enumerate}
\end{lemma}
\proof
For the first assertion, see, eg, \cite[7.3.1.3]{andre:motifs}); for the second, see \cite[5.4.2.2]{andre:motifs}. The third part is the basic motivating consequence of the Standard Conjectures (a fiber functor over $\Q$ is given by Betti cohomology, after fixing an embedding $k \into \CC$): see \cite[Corollary 2]{jannsen}, especially for the semi-simplicity claim. Finally, for the last part, fullness is the Tate conjecture; and faithfulness follows from the agreement of numerical and homological equivalence and \cite[Lemma 2.5]{tate:motives} (note that faithfulness on $\mc{M}_k$ is simply by definition of homological equivalence: it is only with $\Q_{\ell}$-coefficients that some argument is needed).
\endproof  
For the rest of this section, we assume the Tate conjecture for all finitely-generated $k$ of characteristic zero. Thus, we have a motivic Galois formalism: $\mc{M}_{k, E}$ is equivalent to  $\Rep(\mc{G}_{k, E})$ for some pro-reductive group $\mc{G}_{k, E}$ over $E$, the equivalence depending on the choice of an $E$-linear fiber functor. We will implicitly fix an embedding $k \into \CC$ and use the associated Betti realization as our fiber functor.
Before proceeding, we introduce two pieces of notation. For an extensions of fields $k'/k$, we denote the base-change of motives by
\[
(\cdot)|_{k'}\colon \mc{M}_{k, E} \to \mc{M}_{k', E}.
\]
This is not to be confused with the change of coefficients. Fix an embedding $\iota \colon \Qb \into \Qlb$, so that when $E$ is a subfield of $\Qb$ we can speak of the $\ell$-adic realization
\[
H_{\iota} \colon \mc{M}_{k, E} \to \Rep_{\Qlb}(\gal{k})
\]
associated to $\iota$.

Now we turn to the case of number fields, i.e. $k=K$. The Tate conjecture alone does not suffice to link Galois representations with motives: it yields full faithfulness of the $\ell$-adic realization (as in Lemma \ref{motivesrecall}), but does not characterize the essential image. This is done via the combination of the Fontaine-Mazur and Grothendieck-Serre semi-simplicity conjectures, which we now recall. Recall that a semi-simple representation $r_{\ell} \colon G_K \to \mr{GL}_N(\Ql)$ is said to be \textit{geometric} (in the sense of Fontaine-Mazur: \cite{FontaineMazur}) if it is unramified outside a finite set of places of $K$, and if for all $v \vert \ell$ of $K$, the restriction $r_{\ell}|_{G_{K_v}}$ is de Rham (equivalently, potentially semi-stable, as in the original formulation). 
See \cite{Fontaine,ConradH} for the definition and basic properties of deRham representations.
Fontaine and Mazur have conjectured that any irreducible geometric $r_{\ell}$ is isomorphic to a sub-quotient of $H^i(X_{\overline{K}}, \Ql)(j)$ for some smooth projective variety $X/K$ and some integers $i$ and $j$; that the converse assertion holds is a consequence of the base-change theorems of \'{e}tale cohomology (\cite{deligne4.5}) and Faltings' $p$-adic de Rham comparison isomorphism (\cite{faltingsdR}). Grothendieck and Serre have moreover conjectured that for any smooth projective $X/K$, and any integer $i$, $H^i(X_{\overline{K}}, \Ql)$ is a semi-simple representation of $G_K$. Putting all of these conjectures together, we can characterize the essential image of $H_{\iota}$:
\begin{lemma}\label{FM}
Assume the Tate, Fontaine-Mazur, and Grothendieck-Serre conjectures. Let $r_{\ell} \colon \gal{K} \to \mr{GL}_N(\Ql)$ be an irreducible geometric Galois representation. Then there exists an object $M$ of $\mc{M}_{K, \Qb}$ such that
\[
r_{\ell} \otimes_{\Ql} \Qlb \cong H_{\iota}(M).
\]
More generally, the essential image of $H_{\iota}$ consists of all semi-simple geometric representations (with coefficients in $\Qlb$) of $G_K$.
\end{lemma}
\proof
The Fontaine-Mazur conjecture asserts that for some smooth projective variety $X/k$, $r_{\ell}$ is a sub-quotient of $H^i(X_{\overline{K}}, \Ql)(j)$ for some integers $i$ and $j$, and the Grothendieck-Serre conjecture implies this sub-quotient is in fact a direct summand. Under the K\"{u}nneth Standard Conjecture (a consequence of our hypotheses by Lemma \ref{motivesrecall}), $\mc{M}_K$ has a canonical (weight) grading, and we denote by $H^i(X)$ the weight $i$ component of the motive of $X$. The Tate conjecture then implies (Lemma \ref{motivesrecall}) that
\begin{equation}\label{Tate}
H_{\iota} \colon \End_{\mc{M}_K}\left( H^i(X)(j) \right) \otimes_{\Q} \Qlb \xrightarrow{\sim} \End_{\Qlb[\gal{K}]}\left( H^i(X_{\overline{K}}, \Qlb)(j) \right)
\end{equation}
is an isomorphism.

Now, there is a projector (of $\Qlb[\gal{K}]$-modules) $H^i(X_{\overline{K}}, \Qlb)(j) \onto r_{\ell}$, which combined with Equation (\ref{Tate}) yields a projector in $\End_{\mc{M}_K}(H^i(X)(j)) \otimes_{\Q} \Qlb$ whose image has $\ell$-adic realization $r_{\ell}$. But $\End_{\mc{M}_K}(H^i(X)(j))$ is a semi-simple algebra over $\Q$ (Lemma \ref{motivesrecall}), which certainly splits over $\Qb$, so the decomposition of $H^i(X)(j)$ into simple objects of $\mc{M}_{K, \Qlb}$ is already realized in $\mc{M}_{K, \Qb}$.\footnote{In fact, it is realized over the maximal CM subfield of $\Qb$: see e.g. \cite[Lemma 4.1.22]{stp:variationsarxiv}.}

For the final claim about the essential image (which we do not use in what follows), it suffices to show an irreducible $r_{\iota} \colon G_K \to \mr{GL}_N(\Qlb)$ lies in the essential image. Such an $r_{\iota}$ is defined over a finite extension of $\Ql$ and can thus be regarded as a higher-dimensional geometric representation $r_{\ell}$ with $\Ql$-coefficients, necessarily semi-simple. By the first part of the lemma, $r_{\ell} \otimes_{\Ql} \Qlb$ is isomorphic to $H_{\iota}(M)$ for some $M \in \mc{M}_{K, \Qb}$, and by the Tate conjecture there is a projector in $\End(M) \otimes_{\Qb} \Qlb$ inducing the canonical (adjunction) projector $r_{\ell} \otimes_{\Ql} \Qlb \onto r_{\iota}$. Arguing as before (a simple object of $\mc{M}_{K, \Qlb}$ arises by scalar-extension from one of $\mc{M}_{K, \Qb}$), we see that $r_{\iota}$ is in the essential image of $H_{\iota}$.
\endproof
Returning to our particular setting, fix any $\ell_0 \not \in S$ as in our first condition on the compatible system $\{\rho_{\ell}\}_{\ell \not \in S}$, and also fix an embedding $\iota_0 \colon \Qb \into \Qb_{\ell_0}$, so that Lemma \ref{FM} provides us with a number field (the linear combinations of correspondences needed to cut out a given object of $\mc{M}_{K, \Qb}$ have coefficients in a finite extension of $\Q$) $E \subset \Qb$ (which we may assume Galois over $\Q$) and a motivic Galois representation $\rho \colon \mc{G}_{K, E} \to \mr{GL}_{N, E}$ such that $H_{\iota_0}(\rho) \cong \rho_{\ell_0} \otimes \overline{\Q}_{\ell_0}$. Let us denote by $\lambda_0$ the place of $E$ induced by $E \subset \Qb \xrightarrow{\iota_0} \Qlb$. Then for all finite places $\lambda$ of $E$ (say $\lambda \vert \ell$), and for almost all places $v$ of $K$, compatibility gives us the following equality of rational numbers (note that $\rho_{\lambda}$ denotes the $\lambda$-adic realization of the motivic Galois representation $\rho$, while $\rho_{\ell}$ denotes the original $\ell$-adic representation in our compatible system):
\[
\tr(\rho_{\lambda}(fr_v))= \tr(\rho_{\lambda_0}(fr_v))= \tr(\rho_{\ell_0}(fr_v))= \tr(\rho_\ell(fr_v).
\]
Here we use the fact that the collection of $\ell$-adic realizations of a motive form a (weakly) compatible system; this follows from the Lefschetz trace formula, in its `formal' version for correspondences (see for instance \cite[3.3.3, 7.1.4]{andre:motifs}). We deduce as usual (Brauer-Nesbitt and Chebotarev) that $\rho_{\ell}^{\mr{ss}} \otimes_{\Ql} E_{\lambda} \cong \rho_{\lambda}$; this holds for all $\lambda$ for which $\rho_{\ell}$ makes sense, i.e. for all $\lambda$ above $\ell \not \in S$.

Recall that for some $\ell_1 \not \in S$, we have assumed $\rho_{\ell_1}$ is absolutely irreducible; hence for any place $\lambda_1$ of $E$ above $\ell_1$, the previous paragraph shows that $\rho_{\lambda_1} \cong \rho_{\ell_1} \otimes E_{\lambda_1}$ is absolutely irreducible. \textit{A fortiori}, $\rho$ is absolutely irreducible, and then by the Tate conjecture all $\rho_{\lambda}$ are absolutely irreducible, so we can upgrade the conclusion of the previous paragraph to an isomorphism of absolutely irreducible representations $\rho_{\ell} \otimes_{\Ql} E_{\lambda} \cong \rho_{\lambda}$, for all $\ell \not \in S$. 

The next question is whether having each (or almost all) $\rho_{\lambda}$ in fact definable over $\Ql$ forces $\rho$ to be definable over $\Q$. Since the $\rho_{\lambda}$ descend to $\Ql$, the Tate conjecture implies that for all $\sigma \in \Gal(E/\Q)$, ${}^{\sigma} \rho \cong \rho$; and since $\End(\rho)$ is $E$, the obstruction to descending $\rho$ to a $\Q$-rational representation of $\mc{G}_{K}$ is an element $\mr{obs}_{\rho}$ of $H^1(\Gal(E/\Q), \mr{PGL}_N(E))$.
\begin{lemma}
With the notation above, $\mr{obs}_{\rho}$ in fact belongs to
\[
\ker \left( H^1(\Gal(E/\Q), \mr{PGL}_N(E)) \to \prod_{\ell \not \in S} H^1(\Gal(E_{\lambda}/\Q_{\ell}), \mr{PGL}_N(E_{\lambda}) \right).
\]
In particular, if $S$ is empty, then $\rho$ can be defined over $\Q$.
\end{lemma}
\proof
We know that each of the $\lambda$-adic realizations $\rho_{\lambda}$ (for $\lambda \vert \ell \not \in S$) can be defined over $\Ql$; to prove the lemma, we need to verify that the canonical localizations of $\mr{obs}_{\rho}$ (which arise by extending scalars on the motivic Galois representation) are in fact given by the corresponding obstruction classes for the $\lambda$-adic realizations. Thus, we have to recall how these realizations are constructed from $\rho$ itself. The surjection $\mc{G}_{K} \onto \gal{K}$ admits a continuous section on $\Ql$-points, $s_{\ell} \colon \gal{K} \to \mc{G}_K(\Ql)$; composition with $\rho \otimes_E E_{\lambda}$ yields $\rho_{\lambda}$:
\[
\xymatrix{
G_K \ar@/^1.5pc/[rrr]^{\rho_{\lambda}} \ar[r]_-{s_{\ell}} & \mc{G}_K(\Ql) \ar@{^{(}->}[r] & \mc{G}_{K, E}(E_{\lambda}) \ar[r]_{\rho \otimes_E E_{\lambda}} & \mr{GL}_N(E_{\lambda}).
}
\]
By construction of the respective obstruction classes, the canonical map from endomorphisms of $\rho \otimes_E E_{\lambda}$ to those of $\rho_{\lambda}$ realizes the obstruction class for $\rho_{\lambda}$ as the localization of $\mr{obs}_{\rho}$ at $\Gal(E_{\lambda}/\Ql)$. But we have seen that $\rho_{\lambda}$ can be defined over $\Ql$, so we conclude that
$\mr{obs}_{\rho}$ has trivial restriction to each $\Gal(E_{\lambda}/\Ql)$, as desired.

For the final claim, note that by Hilbert 90 we can regard $\mr{obs}_{\rho}$ as an element of
\[
\ker \left( H^2(\Gal(E/\Q), E^\times) \to \prod_{\ell \not \in S} H^2(\Gal(E_{\lambda}/\Ql), E_{\lambda}^\times) \right).
\]
If $S$ is empty, then the structure of the Brauer group of $\Q$ (which has only one infinite place!) then forces $\mr{obs}_{\rho}$ to be trivial.
\endproof
\proof[Proof of Theorem \ref{main}] From now on we assume $S= \emptyset$, so that our compatible system $\{\rho_{\ell}\}_{\ell}$ arises from a rational representation
\[
\rho \colon \mc{G}_K \to \mr{GL}_{N, \Q}.
\]
Let $M$ be the rank $N$ object of $\mc{M}_K$ corresponding to $\rho$ via the Tannakian equivalence. Recall that we are given a prime $\ell_2$ and a place $v \vert \ell_2$ of $K$ for which we are given that $\rho_{\ell_2}|_{\gal{K_v}}$ is de Rham with Hodge numbers equal to those of an abelian variety of dimension $\frac{N}{2}$. All objects of $\mc{M}_K$ enjoy the de Rham comparison theorem of `$\ell_2$-adic Hodge theory': denoting Fontaine's period ring over $K_v$ by $\mr{B}_{\mr{dR}, K_v}$, and the de Rham realization functor by $H_{\mr{dR}} \colon \mc{M}_K \to \Fil_K$ (the category of filtered $K$-vector spaces), we have the comparison (respecting filtration and $\gal{K_v}$-action)
\[
H_{\mr{dR}}(M) \otimes_K \mr{B}_{\mr{dR}, K_v} \xrightarrow{\sim} H_{\ell_2}(M) \otimes_{\Q_{\ell_2}} \mr{B}_{\mr{dR}, K_v},
\]
hence
\[
H_{\mr{dR}}(M) \otimes_K K_v \cong \mr{D}_{\mr{dR}, K_v}(H_{\ell_2}(M)).
\]
The Hodge filtration on $H_{\mr{dR}}(M)$ therefore satisfies
\begin{equation}\label{gr}
\dim_K \gr^0 \left( H_{\mr{dR}}(M) \right) = \dim_K \gr^{-1} \left( H_{\mr{dR}}(M) \right)= \frac{N}{2}
\end{equation}
and $\gr^i \left( H_{\mr{dR}}(M) \right)=0$ for $i \neq 0, -1$.

Now we turn to the Betti picture. Recall that to define the fiber functor on $\mc{M}_K$ we had to fix an embedding $K \into \CC$; we regard $K$ as a subfield of $\CC$ via this embedding. Then we also have the analytic Betti-de Rham comparison isomorphism
\begin{equation}\label{B-dR}
H_{\mr{dR}}(M) \otimes_K \CC \xrightarrow{\sim} H_{\mr{B}}(M|_{\CC}) \otimes_{\Q} \CC.
\end{equation}
We collect our findings in the following lemma, which relies on an application of the Hodge conjecture:
\begin{lemma}\label{hodge}
There is an abelian variety $A$ over $K$, and an isomorphism of motives $H_1(A) \cong M$.
\end{lemma}
\proof
We see from Equations (\ref{gr}) and (\ref{B-dR}) that $H_{\mr{B}}(M|_{\CC})$ is a polarizable rational Hodge structure of type $\{(0, -1), (-1, 0)\}$. It follows from Riemann's theorem that there is an abelian variety $A/\CC$ and an isomorphism of $\Q$-Hodge structures $H_1(A(\CC), \Q) \cong H_{\mr{B}}(M|_{\CC})$. The Hodge conjecture implies that this isomorphism comes from an isomorphism $H_1(A) \xrightarrow{\sim} M|_{\CC}$ in $\mc{M}_{\CC}$.

For any $\sigma \in \Aut(\CC/\Qb)$, we deduce an isomorphism
\[
{}^{\sigma} H_1(A) \xrightarrow{\sim} {}^{\sigma} M|_{\CC} = M|_{\CC} \xleftarrow{\sim} H_1(A),
\]
and again from Riemann's theorem we see that ${}^{\sigma} A$ and $A$ are isogenous.

The following statement will be proven later in this paper.
\begin{lemma}
\label{isogenyCount}
Let $\calK$ be a countable subfield of the field $\CC$ and
$\bar{\calK}$ the algebraic closure of $\calK$ in $\CC$.
Let $\calA$ be a
complex abelian variety of  dimension $g$ such that for each field automorphism $\sigma
\in \Aut(\CC/\calK)$ the complex abelian varieties $\calA$  and its ``conjugate" $^{\sigma}\calA=\calA\times_{\CC,\sigma}\CC$   are
isogenous. Then there exists an abelian variety $\calA_0$ over $\bar{\calK}$  such that
$\calA_0\times_{\bar{\calK}} \CC$ is isomorphic to $\calA$.

\end{lemma}

It follows from Lemma \ref{isogenyCount} that
 $A$ has a model $A_{\Qb}$ over $\Qb$. The morphism
\[
\Hom_{\mc{M}_{\Qb}}(H_1(A_{\Qb}), M|_{\Qb}) \to \Hom_{\mc{M}_{\CC}}(H_1(A), M|_{\CC})
\]
is an isomorphism, and then by general principles we deduce the existence of some finite extension $L/K$ inside $\Qb$ over which $A$ descends to an abelian variety $A_L$, and where we have an isomorphism $H_1(A_L) \xrightarrow{\sim} M|_{L}$  in $\mc{M}_L$.

Finally, we treat the descent to $K$ itself. We form the restriction of scalars abelian variety $\Res_{L/K}(A_L)$; under the \textit{fully faithful} embedding
\begin{align*}
\mr{AV}^0_K &\subset \mc{M}_K \\
 B &\mapsto H_1(B),
\end{align*}
we can think of $H_1(\Res_{L/K}(A_L))$ as $\Ind_{L}^K (H_1(A_L))$, where the induction is taken in the sense of motivic Galois representations (note that the quotient $\mc{G}_K/\mc{G}_L$ is canonically $\Gal(L/K)$, so this is just the usual induction from a finite-index subgroup). Frobenius reciprocity then implies the existence of a non-zero map $M \to \Ind_L^K (H_1(A_L))$ in $\mc{M}_K$. Since $M$ is a simple motive, this map realizes it as a direct summand in $\mc{M}_K$, and consequently (full-faithfulness) in $\mr{AV}^0_K$ as well. That is, there is an endomorphism of $\Res_{L/K}(A_L)$ whose image is an abelian variety $A$ over $K$ with $H_1(A) \cong M$.
\endproof

{\bf Proof of Lemma \ref{isogenyCount}}. We may assume that $g\ge 1$.
Since $\bar{\calK}$ is also countable, we may replace $\calK$ by $\bar{\calK}$, i.e., assume that $\calK$ is algebraically closed. Since the isogeny class of $\calA$ consists of a countable set of (complex) abelian varieties (up to an isomorphism), we conclude that the set $\Aut(\CC/\calK)(\calA)$ of isomorphism classes  of complex abelian varieties of the form
$\{^\sigma \calA\mid \sigma \in \Aut(\CC/\calK)\}$ is either finite or countable.

 Our plan is as follows. Let us consider a  {\sl fine} moduli space $\A_{g,?}$ over $\Qb$ of $g$-dimensional abelian varieties (schemes)   with certain additional structures (there should be only finitely many choices of these structures for any given abelian variety) such that it is a quasiprojective subvariety  in some projective space ${\mathbf P}^N$. Choose these additional structures for $\calA$ (there should be only finitely many choices) and let $P \in \A_{g,?}(\CC)$ be the corresponding point of our moduli space. We need to prove that
$$P \in \A_{g,?}(\calK).$$
Suppose that it is not true. Then the orbit $\Aut(\CC/\calK)(P)$ of $P$ is {\sl uncountable}.
Indeed, $P$ lies in one of the $(N+1)$ affine charts/spaces ${\mathbf A}^N$ that do cover ${\mathbf P}^N$. This implies that $P$ does {\sl not} belong to ${\mathbf A}^N(\calK)$ and therefore (at least) one of its coordinates is transcendental over $\calK$. But the $\Aut(\CC/\calK)$-orbit of this coordinate coincides with uncountable  $\CC\setminus \calK$ and therefore the $\Aut(\CC/\calK)$-orbit $\Aut(\CC/\calK)(P)$ of $P$ is uncountable in $\A_{g,?}(\CC)$. However, for each $\sigma \in \Aut(\CC/\calK)$
the point $\sigma(P)$ corresponds to $^{\sigma} \calA$ with some additional structures and there are only finitely many choices for these structures. Since we know that the orbit $\Aut(\CC/\calK)(\calA)$ of $\calA$, is, at most, countable, we conclude that the orbit $\Aut(\CC/\calK)(P)$ of $P$ is also, at most, countable, which is not the case.  This gives us a desired contradiction.

We choose as $\A_{g,?}$ the moduli space of (polarized) abelian schemes of relative dimension $g$ with theta structures of type $\delta$ that was introduced and studied by D. Mumford \cite{MumfordEquations}. In order to choose (define) a suitable $\delta$, let us pick
 a totally symmetric ample invertible sheaf $\calL_0$ on $\calA$ \cite[Sect. 2]{MumfordEquations} and consider its $8$th power $\calL:=\calL_0^8$ in $\Pic(\calA)$. Then $\calL$ is a very ample invertible sheaf that defines a polarization $\Lambda(\calL)$ on $\calA$ \cite[Part I, Sect. 1]{MumfordEquations} that is an isogeny from $\calA$ to its dual; the kernel $H(\calL)$ of $\Lambda(\calL)$ is a finite commutative subgroup of $\calA(\CC)$ (that contains all points of order $8$). The order  of $H(\calL)$ is the degree of the polarization. The type $\delta$ is essentially the isomorphism class of the group $H(\calL)$ \cite[Part I, Sect. 1, p. 294]{MumfordEquations}. The resulting moduli space $\A_{g,?}:=M_{\delta}$ \cite[Part II, Sect. 6]{MumfordEquations} enjoys all the properties that we used in the course of the proof.
\endproof
Here is the anabelian application already mentioned in the introduction:
\begin{thm}
\label{coro}
Assume the Hodge, Tate, Fontaine-Mazur, and Grothendieck-Serre conjectures.
Suppose $s \in S_0(K,\A_g)$ gives rise to a system of $\ell$-adic Galois representations one of which is absolutely irreducible. Then there exists up to isomorphism a unique abelian variety $B/K$ with $\sigma_{\A_g/K}(B)=s$.
\end{thm}

{\bf Proof.}
Let us write $s_{\ell}$ for the $\ell$-adic representation associated to $s$; thus $s_{\ell}$ is a representation of $G_K$ on a free $\Z_{\ell}$-module of rank $2g$. 
The hypotheses of Theorem \ref{main} apply to the compatible system 
associated to $s$ by assumption,
so we obtain an abelian variety $A/K$ (well-defined up to isogeny) whose rational Tate modules $V_{\ell}(A)$ are isomorphic to the given $s_{\ell} \otimes_{\Z_{\ell}} \Q_{\ell}$ (for all $\ell$). Moreover, the hypotheses also imply that the endomorphism ring of $A$ is $\Z$. It remains to see that within the isogeny class of $A$ there is an abelian variety $B$ over $K$ whose integral Tate modules $T_{\ell}(B)$ are isomorphic to the $s_{\ell}$ (as $\Z_{\ell}$-representations), i.e. such that $\sigma_{\A_g/K}(B)=s$. For this, we first observe that by \cite[Proposition 3.3]{DelBourbaki} (which readily generalizes to abelian varieties of any dimension), it suffices to show that for almost all $\ell$, there is an isomorphism $T_{\ell}(A) \cong s_{\ell}$. Since $\End(A)= \Z$, \cite[Corollary 5.4.5]{Zar_Inv85} implies that $A[\ell]$ is absolutely simple for almost all $\ell$, and hence that for almost all $\ell$, all Galois-stable lattices in $V_{\ell}(A)$ are of the form $\ell^m T_{\ell}(A)$ for some integer $m$; we conclude that $T_{\ell}(A)$ is isomorphic to $s_{\ell}$ for almost all $\ell$. Thus there exists an abelian variety $B$ in the isogeny class of $A$ such that $\sigma_{\A_g/K}(B)= s$.  
 In order to prove the uniqueness of such a $B$ up to an isomorphism,  first, notice that $\End(B)= \Z$. Second, let $C$ be an abelian variety over $K$ such that the $\Z_{\ell}[G_K]$-modules $T_{\ell}(B)$ and $T_{\ell}(C)$ are isomorphic for all primes $\ell$. This implies that the $\Z_{\ell}$-ranks of  $T_{\ell}(B)$ and $T_{\ell}(C)$ coincide and therefore
 $$\dim(B)=\dim(C).$$
  By a theorem of Faltings \cite{Faltings},
 $$\Hom(B,C)=\Hom_{G_K}(T_{\ell}(B), T_{\ell}(C)).$$
 Since $\Hom(B,C)$ is dense in $\Hom(B,C)\otimes \Z_{\ell}$ in the $\ell$-adic topology, and the set of isomorphisms $T_{\ell}(B)\cong T_{\ell}(C)$ is open in $\Hom(B,C)\otimes \Z_{\ell}$, there  is a homomorphism  $\phi_{\ell}\in \Hom(B,C)$ that induces an isomorphism of Tate modules $T_{\ell}(B)\cong T_{\ell}(C)$. Clearly, $\ker(\phi_{\ell})$ does {\sl not} contain points of order $\ell$ and therefore is finite.  Since $\dim(B)=\dim(C)$, we obtain that $\phi_{\ell}$ is an isogeny, whose degree is prime to $\ell$. In particular, $B$ and $C$ are isogenous. On the other hand, since $\End(B)=\Z$, the group $\Hom(B,C)$ is a free $\Z$-module of rank $1$. 
 Let us choose $\psi: B \to C$ that is a generator of $\Hom(B,C)$.  Clearly, $\psi$ is an isogeny. Since for {\sl all primes} $\ell$
 $$\phi_{\ell}\in \Hom(B,C)=\Z\cdot \psi,$$
 $\deg(\psi)$ is {\sl not} divisible by $\ell$ and therefore $\deg(\psi)=1$, i.e., $\psi$ is an isomorphism of abelian varieties $B$ and $C$.


\endproof

Results in the same vein as this corollary have been obtained for elliptic curves over $\Q$
in \cite{HelmVoloch} 
 and for elliptic curves over function fields in \cite{Voloch}.

\section{Counterexamples}
\label{counterE}
Now we will construct an example of Galois representation that will provide
us with examples that show that some of the hypotheses of the above results
are indispensable.


Let $k$ be a real quadratic field. Let us choose a prime $p$
that splits in $k$. Now let $D$ be the indefinite quaternion
$k$-algebra that  splits everywhere outside (two)  prime divisors of
$p$ and is ramified at these divisors.   If  $\ell$ is a prime  then we have
$$D\otimes_{\Q}\Q_{\ell}=[D\otimes_k k]\otimes_{\Q}\Q_{\ell} =
D\otimes_k [k\otimes_{\Q}\Q_{\ell}].$$
This implies that if $\ell\ne p$  then
$D\otimes_{\Q}\Q_{\ell}$ is either (isomorphic to) the {\sl simple} matrix algebra (of size 2) over a quadratic extension of $\Q_{\ell}$ or a direct sum of two copies of of the {\sl simple} matrix algebra (of size 2) over $\Q_{\ell}$. (In both cases, $D\otimes_{\Q}\Q_{\ell}$ is isomorphic to the  matrix algebra $\MM_2(k\otimes_{\Q}\Q_{\ell})$ of size 2 over $k\otimes_{\Q}\Q_{\ell}$.)

 In particular, the image of $D\otimes_{\Q}\Q_{\ell}$ under each nonzero $\Q_{\ell}$-algebra homomorphism contains {\sl zero divisors}.


Let $Y$ be an abelian variety over a field $L$. Suppose that all $\bar{L}$-endomorphisms of $Y$ are defined over $L$ and there is a $\Q$-algebra embedding
$$D \hookrightarrow \End^0(Y)$$
that sends $1$ to $1$.  This gives us the embedding
$$D\otimes_{\Q}\Q_{\ell}\subset \End^0(Y)\otimes_{\Q}\Q_{\ell}\subset
\End_{G_L}(V_{\ell}(Y)).$$
Recall that if $\ell \ne p$ then $D\otimes_{\Q}\Q_{\ell}$ is isomorphic to the matrix algebra of size $2$ over  $k\otimes_{\Q}\Q_{\ell}$. This implies that there are two isomorphic $\Q_{\ell}[G_L]$-submodules $W_{1,\ell}(Y)$ and  $W_{2,\ell}(Y)$ in $V_{\ell}(Y)$ such that
$$V_{\ell}(Y)= W_{1,\ell}(Y)\oplus  W_{2,\ell}(Y)\cong
W_{1,\ell}(Y)\oplus W_{1,\ell}(Y)\cong W_{2,\ell}(Y)\oplus W_{2,\ell}(Y)
.$$
If we denote by  $W_{\ell}(Y)$ the $\Q_{\ell}[G_L]$-module $W_{1,\ell}$ then we get an isomorphism of $\Q_{\ell}[G_L]$-modules
$$V_{\ell}(Y)\cong W_{\ell}(Y)\oplus W_{\ell}(Y).$$
This implies that the centralizer
$\End_{G_L}(V_{\ell}(Y))$ coincides with the matrix algebra $\MM_2(\End_{G_L}(W_{\ell}(Y)))$ of size $2$ over the centralizer
$\End_{G_L}(W_{\ell}(Y))$.

If $\ell=p$ then $k\otimes_{\Q}\Q_p=\Q_p\oplus\Q_p$  and $D\otimes_{\Q}\Q_{p}$ splits into a direct sum of two (mutually isomorphic) quaternion algebras over $\Q_p$. This also gives us a splitting of the Galois module $V_p(Y)$ into a direct sum
$$V_p(Y)= W_{1,p}(Y)\oplus  W_{2,p}(Y).$$
of its certain nonzero  $\Q_{p}[G_L]$-submodules $W_{1,p}(Y)$ and $W_{2,p}(Y)$. (Actually,
$$\dim_{\Q_p} W_{1,p}=\dim_{\Q_p} W_{2,p}=\dim(Y),$$
because $V_p(Y)$ is a {\sl free} $k\otimes_{\Q}\Q_p$-module of rank $2\dim(Y)/[k:\Q]=\dim(Y)$ \cite[Th. 2.1.1 on p. 768]{Ribet}.)

\vskip .5cm

{\bf Remark} Let $L$ be a finitely generated field of characteristic $0$.
Suppose that $D=\End^0(Y)$. By  Faltings' results about the Galois action on Tate modules of abelian varieties \cite{Faltings,Faltings2}, the $G_L$-module  $V_{\ell}(Y)$ is semisimple and
$$\End_{G_L}(V_{\ell}(Y))=D\otimes_{\Q}\Q_{\ell}.$$
This implies that if $\ell \ne p$ then (the submodule) $W_{\ell}(Y)$ is also semisimple and
$$\MM_2(\End_{G_L}(W_{\ell}(Y)))\cong \MM_2(k\otimes_{\Q}\Q_{\ell}).$$
It follows that
$$\End_{G_L}(W_{\ell}(Y))\cong k\otimes_{\Q}\Q_{\ell}.$$
On the other hand,
the $G_L$-modules $W_{1,p}(Y)$ and  $W_{2,p}(Y)$ are non-isomorphic.

\vskip .5cm

 According to Shimura (\cite{Shimura}, see also the case of Type II($e_0=2$) with $m=1$ in \cite[Table 8.1 on p. 498]{OortEndo} and \cite[Table on p. 23]{OZ}) there exists a complex abelian fourfold $X$, whose endomorphism algebra $\End^0(X)$ is isomorphic to $D$. Clearly, $X$ is defined over a finitely generated field of characteristic zero.  It follows from Serre's variant of Hilbert's
irreducibility theorem for infinite Galois extensions  combined with results of Faltings that there exists a number field $K$ and an abelian fourfold $A$ over $K$ such that the
endomorphism algebra  of all $\bar{K}$-endomorphisms of $A$ is also isomorphic to $D$ (see \cite[Cor. 1.5 on p. 165]{noot}). Enlarging $K$, we may assume that all points of order $12$ on $A$ are defined over $K$ and therefore, by the theorem of Silverberg,  all $\bar{K}$-endomorphisms of $A$ are defined over $K$. Now Raynaud's criterion (\cite{Gro}, see also \cite{SZ}) implies that $A$ has everywhere semistable reduction. On the other hand,
$$\dim_{\Q}\End^0(A)=\dim_{\Q}D=8>4=\dim(A).$$
By \cite[Lemma 3.9 on p. 484]{OortEndo}, $A$ has everywhere potential good reduction. This implies that $A$ has good reduction everywhere. If $v$ is a nonarchimedean place of $K$ with finite residue field $\kappa(v)$ then we write $A(v)$ for the reduction of $A$ at $v$; clearly, $A(v)$ is an abelian fourfold over $\kappa(v)$. If $\fchar(\kappa(v)) \ne 2$ then all points of order $4$ on $A(v)$ are defined over $\kappa(v)$; if $\fchar(\kappa(v)) \ne 3$ then all points of order $3$ on $A(v)$ are defined over $\kappa(v)$. It follows from the theorem of Silverberg   that all $\overline{\kappa(v)}$-endomorphisms of $A(v)$ are defined over $\kappa(v)$. 
 For each $v$ we get an embedding of $\Q$-algebras
$$D \cong \End^0(A) \hookrightarrow \End^0(A(v)).$$
In particular,  $\End^0(A(v))$ is a {\sl noncommutative} $\Q$-algebra, whose $\Q$-dimension is divisible by $8$.


\begin{thm}
\label{nonsimple}
If $\ell:=\fchar(\kappa(v)) \ne p$ then $A(v)$ is not simple over $\kappa(v)$.
\end{thm}

\begin{proof}
We write $q_v$ for the cardinality of $\kappa(v)$. Clearly, $q_v$ is a power of $\ell$.

Suppose that  $A(v)$ is  simple over $\kappa(v)$.  Since all endomorphisms of $A(v)$  are defined over $\kappa(v)$, the abelian variety  $A(v)$ is  {\sl absolutely simple}.

Let $\pi$ be a {\sl Weil $q_v$-number} that corresponds to the $\kappa(v)$-isogeny class of $A(v)$ \cite{tateendff,TateBourbaki}. In particular,  $\pi$ is an algebraic integer (complex number), all whose Galois conjugates have (complex) absolute value $\sqrt{q_v}$. In particular, the product
$$\pi \bar{\pi}=q_v,$$
where $\bar{\pi}$ is the complex conjugate of $\pi$.

Let $E=\Q(\pi)$ be the number field generated by $\pi$
and let $\Oc_E$ be the ring of integers in $E$.
Then $E$ contains $\bar{\pi}$
and is isomorphic to the center of $\End^0(A(v))$ \cite{tateendff,TateBourbaki};  one may view $\End^0(A(v))$ as a {\sl central} division algebra over $E$. It is known that
$E$ is either $\Q$, $\Q(\sqrt{\ell})$ or a (purely imaginary) CM field \cite[p. 97]{TateBourbaki}.  It is known (ibid) that in the first two (totally real) cases
 simple $A(v)$ has dimension $1$ or $2$, which is not the case. So, $E$ is a CM field;   Since $\dim(A(v))=4$ and $[E:\Q]$ divides $2\dim(A(v))$, we have
$[E:\Q]=2, 4$ or $8$.
By \cite[p. 96, Th. 1(ii), formula (2)]{TateBourbaki}
\footnote{In \cite{TateBourbaki} our $E$ is denoted by $F$ while our $\End^0(A(v))$ is denoted by $E$.},
$$8=2\cdot 4=2\dim(A(v)))=\sqrt{\dim_E(\End^0(A(v))}\cdot [E:\Q].$$
Since $\End^0(A(v))$ is  {\sl noncommutative}, it follows that $E$ is either an imaginary quadratic field and  $\End^0(A(v))$ is a $16$-dimensional division algebra over $E$ or $E$ is a CM field of degree $4$ and $\End^0(A(v))$ is a $4$-dimensional (i.e.,  quaternion) division  algebra over $E$.  In both cases $\End^0(A(v))$ is unramified at all places of $E$ except  some places of residual characteristic $\ell$ \cite[p. 96, Th. 1(ii)]{TateBourbaki}.  It follows from the Hasse--Brauer-Noether theorem that
 $\End^0(A(v))$ is ramified at, at least, two places of $E$ with residual characteristic $\ell$. This implies that $\Oc_E$ contains, at least, two maximal ideals that lie above $\ell$.

Clearly,
$$\pi, \bar{\pi} \in \Oc_E.$$
Recall that $\pi \bar{\pi}=q_v$ is a power of $\ell$. This implies that
for every prime $r \ne \ell$ both $\pi$  and $\bar{\pi}$ are $r$-adic units in $E$.

First assume that $E$ has degree $4$ and $\End^0(A(v))$ is a quaternion algebra. Then (thanks to the theorem of Hasse--Brauer--Noether)  there exists a place $w$ of $E$ with residual characteristic  $\ell$ and  such that the localization
$\End^0(A(v))\otimes_E E_w$ is a quaternion division algebra over the $w$-adic field $E_w$. On the other hand, there is a nonzero (because it sends $1$ to $1$)  $\Q_{\ell}$-algebra homomorphism
$$D \otimes_{\Q}\Q_{\ell} \to \End^0(A(v))\otimes_{\Q} \Q_{\ell} \twoheadrightarrow \End^0(A(v))\otimes_E E_w.$$ This implies that $\End^0(A(v))\otimes_E E_w$ contains zero divisors, which is not the case and we get a contradiction.

So, now we assume that $E$ is an {\sl imaginary quadratic} field
and
$$\dim_E(\End^0(A(v)))=16=4^2.$$
In particular, the order of the class of $\End^0(A(v))$ in the Brauer group of $E$ divides $4$ and therefore is either $2$ or $4$.

We have already seen that there exist, at least, two maximal ideals
 in $\Oc_E$ that lie above $\ell$. Since $E$ is an imaginary quadratic field,
the ideal  $\ell \Oc_L$  of $ \Oc_L$ splits into a product of two distinct complex-conjugate maximal ideals $w_1$ and $w_2$ and therefore
$$E_{w_1}=\Q_{\ell}, \ E_{w_2}=\Q_{\ell};
\  [E_{w_1}:\Q_{\ell}]=[E_{w_2}:\Q_{\ell}]=1.$$
Let
$$\mathrm{ord}_{w_i}: E^{*} \twoheadrightarrow \Z$$
be the discrete valuation map that corresponds to  $w_i$.
Recall that $q_v$ is a power of $\ell$, i.e., $q_v=\ell^N$ for a certain positive integer $N$. Clearly
$$\mathrm{ord}_{w_i}(\ell)=1, \ \mathrm{ord}_{w_i}(\pi)+ \mathrm{ord}_{w_i}(\bar{\pi})=
 \mathrm{ord}_{w_i}(q_v)=N.$$
By \cite[page 96, Th. 1(ii), formula (1)]{TateBourbaki}, the local invariant of $\End^0(A(v))$ at $w_i$ is
$$\frac{\mathrm{ord}_{w_i}(\pi)}{\mathrm{ord}_{w_i}(q_v)} \cdot  [E_{w_i}:\Q_{\ell}] (\bmod 1)=
\frac{\mathrm{ord}_{w_i}(\pi)}{N} (\bmod 1).$$
In addition,
the sum in $\Q/\Z$ of local invariants  of  $\End^0(A(v))$ at $w_1$ and $w_2$ is zero \cite[Sect. 1, Theorem 1 and Example b)]{TateBourbaki}; we have already seen that its  local invariants at all other places of $E$ do vanish.
Using the Hasse--Brauer-Noether theorem and taking into account that
the order of the class of $\End^0(A(v))$ in the Brauer group of $E$  is either $2$ or $4$, we conclude that the local invariants  of $\End^0(A(v))$ at
 $\{w_1,w_2\}$ are  either $\{1/4 \bmod 1, 3/4 \bmod 1\}$ or
$\{3/4 \bmod 1, 1/4 \bmod 1\}$ (and in both cases the order of  $\End^0(A(v))$ in the Brauer group of $E$ is $4$) or $\{1/2 \bmod 1, 1/2 \bmod 1\}$. In the latter case  it follows from the formula for the $w_i$-adic invariant of $\End^0(A(v))$
 that
$$\mathrm{ord}_{w_i}(\pi)=\frac{N}{2}=\mathrm{ord}_{w_i}(\bar{\pi})$$
and therefore
$\bar{\pi}/\pi$ is a $w_i$-adic unit for both $w_1$ and $w_2$.
Therefore $\bar{\pi}/\pi$ is an $\ell$-adic unit. This implies that $\bar{\pi}/\pi$ is a unit in imaginary quadratic $E$ and therefore is a root of unity.
It follows  that
$$\frac{\pi^2}{q_v}=\frac{\pi^2}{\pi\bar{\pi}}=\frac{\pi}{\bar{\pi}}$$ is   a root of unity. This implies that there is a positive (even) integer $m$ such that
$$\pi^m=q_v^{m/2}\in \Q$$
and therefore $\Q(\pi^m)=\Q$.
 Let $\kappa(v)_m$ be the finite degree $m$ field extension of $\kappa(v)$, which consists of $q_v^m$ elements. Then $\pi^m$ is the Weil $q_v^m$-number that corresponds to the simple $4$-dimensional   abelian variety
$A(v)\times \kappa(v)_m$ over $\kappa(v)_m$. Since $\Q(\pi^m)=\Q$, we conclude (as above) that $A(v)\times \kappa(v)_m$  has dimension $1$ or $2$, which is not the case.

In both remaining cases the order of the algebra $\End^0(A(v))\otimes_E E_{w_1}$ in the Brauer group of the $E_{w_1}\cong \Q_{\ell}$ is $4$. This implies that
$\End^0(A(v))\otimes_E E_{w_1}$ is neither the matrix algebra of size 4 over $E_{w_1}$ nor the matrix algebra of size two over a quaternion algebra over $E_{w_1}$. The only remaining possibility is that $\End^0(A(v))\otimes_ E E_{w_1}$ is a {\sl division algebra} over $E_{w_1}$. However,  there is again  a nonzero (because it sends $1$ to $1$)  $\Q_{\ell}$-algebra homomorphism
$$D \otimes_{\Q}\Q_{\ell} \to \End^0(A(v))\otimes_{\Q} \Q_{\ell} \twoheadrightarrow \End^0(A(v))\otimes_E E_{w_1}.$$ This implies that $\End^0(A(v))\otimes_E E_{w_1}$ contains zero divisors, which is not the case and we get a contradiction.
\end{proof}

\begin{thm}
\label{nonsimple2}
If $\ell:=\fchar(\kappa(v)) \ne p$ then there exists an abelian surface $B(v)$ over  $\kappa(v)$ such that  $A(v)$ is $\kappa(v)$-isogenous to the square $B(v)^2$ of $B(v)$.
\end{thm}

\begin{proof}
We  know that $A(v)$ is {\sl not} simple and that all $\overline{\kappa(v)}$-endomorphisms of $A(v)$ are defined over $k(v)$.
Now let us split $A(v)$ up to a $\kappa(v)$-isogeny into a product of its
$\kappa(v)$-isotypic components, using Poincar\'e Complete Reducibility Theorem \cite[Th. 6 on p. 28 and Th. 7 on p. 30]{Lang}.
In other words, there is a $\kappa(v)$-isogeny
$$\mathcal{S}: \prod_{i\in I}A_i \to A(v)$$
where each $A_i$ is a nonzero abelian $\kappa(v)$-subvariety in $A$
such that $\End^0(A_i)$ is a {\sl simple} $\Q$-algebra and $\mathcal{S}$
induces an isomorphism  of $\Q$-algebras
$$\End^0(A(v)) \cong \End^0(\prod_{i\in I}A_i)=\oplus_{i \in I} \End^0(A_i).$$
This gives us  nonzero $\Q$-algebra homomorphisms
$$D \to \End^0(A_i)$$
that must be injective, since $D$ is a {\sl simple} $\Q$-algebra. This implies that each $\End^0(A_i)$ is a noncommutative simple $\Q$-algebra, whose $\Q$-dimension is divisible by $8$. In particular,  all $\dim(A_i) \ge 2$ and therefore $I$ consists of, at most, $2$ elements, since
$$\sum_{i\in I} \dim(A_i)=\dim(A(v))=4.$$

Since all $\overline{\kappa(v)}$-endomorphisms of $A(v)$ are defined over $k(v)$, all $\overline{\kappa(v)}$-endomorphisms of $A_i$ are also defined over $\kappa(v)$; in addition,  if $i$ and $j$ are distinct elements of $I$ then every $\overline{\kappa(v)}$-homomorphism between $A_i$ and $A_j$ is $0$.

If we have $\dim(A_i)=2$ for some $i$ then either $A_i$ is isogenous
to a square of a supersingular elliptic curve or $A_i$ is an
absolutely simple abelian surface.  However, each absolutely simple
abelian surface over a finite field is either {\sl ordinary} (i.e.,
the slopes of its Newton polygon are $0$ and $1$, both of length
$2$) or {\sl almost ordinary}  (i.e., the slopes of its Newton
polygon are $0$ and $1$, both of length $1$, and $1/2$ with length
2): this assertion is well known and  follows easily from
\cite[Remark 4.1 on p. 2088]{ZarhinJPAA2015}. However, in both
(ordinary and almost ordinary) cases the endomorphism algebra of a
simple abelian variety is commutative \cite[Lemma 2.3 on p.
136]{OortCM}. This implies that if $\dim(A_i)=2$ then $A_i$ is
$\kappa(v)$-isogenous to a square of a supersingular elliptic curve.
However, if $I$ consists of two elements say, $i$ and $j$ then it
follows that both $A_i$ and $A_j$ are $2$-dimensional and therefore
both  isogenous to a square of a supersingular elliptic curve.  This
implies that $A_i$ and $A_j$ are isotypic and therefore $A$ itself
is isotypic and we get a contradiction, i.e., none of $A_i$ has
dimension $2$. It is also clear that if $\dim(A_i)=3$ then
$\dim(A_j)=1$, which could not be the case. This implies that $A(v)$
itself is isotypic.   It follows that if
$\ell=\fchar(\kappa(v))\ne p$ then  $A(v)$ is $\kappa(v)$-isogenous
either to a $4$th power of an elliptic curve or to a square of an
abelian surface over $\kappa(v)$ (recall that $A(v)$ is not
simple!). In both cases   there exists an abelian surface $B(v)$
over $\kappa(v)$, whose square $B(v)^2$ is $\kappa(v)$-isogenous to
$A(v)$.   
\end{proof}

Let $B(v)$ be as in Theorem \ref{nonsimple2}.
One may lift  the abelian surface $B(v)$ over $\kappa(v)$ to an abelian surface $B^{v}$ over
$K_v$, whose reduction is $B(v)$ (see \cite[Prop. 11.1 on p.
177]{OortLift}).  Now if one restricts the action of $G_K$ on the
$\Q_r$-Tate module (here $r$ is any prime different from
$\fchar(\kappa(v)$)
$$V_r(A)=T_r(A)\otimes_{\Z_r} \Q_r$$
to the decomposition group $D(v)=G_{K_v}$ then the corresponding
$G_{K_v}$-module $V_r(A)$ is {\sl unramified} (i.e., the inertia group acts trivially) and isomorphic to
$$V_r(B^{v})\oplus V_r(B^{v}).$$

\begin{thm}
If $r \ne p$ and $\fchar(\kappa(v)) \ne r$ then the $G_{K_v}$-modules
$V_r(B^{v})$ and $W_r(A)$ are isomorphic. In particular, the $G_{K_v}$-modules
$$V_r(A)=W_r(A)\oplus W_r(A)$$
and $$V_r(B^{v})\oplus V_r(B^{v})=V_r((B^{v})^2)$$
are isomorphic.
\end{thm}
\begin{proof}
We know that the $G_{K_v}$-modules $W_r(A)\oplus W_r(A)$ and
 $$V_r(B^{v})\oplus V_r(B^{v})$$ are both  isomorphic to $V_r(A)$. Since the Frobenius endomorphism of $A(v)$ acts on $V_r(A)$ as a semisimple linear operator (by a theorem of A. Weil),  the
 $G_{K_v}$-module $V_r(A)$ is semisimple. This implies that the
$G_{K_v}$-modules
$V_r(B^{v})$ and $W_r(A)$ are isomorphic.
\end{proof}

For primes $\ell \ne p$, the algebra $D \otimes_{\Q} \Q_{\ell}$ splits and correspondingly,
the representation $V_{\ell}(A)$ splits as $W_{\ell} \oplus W_{\ell}$.
Locally, at a place $v \nmid \ell$,
we have $W_{\ell} \cong  V_{\ell}(B^v)$.
 However, {\sl globally}, the representation $W_{\ell}$ does {\sl not} come from an
abelian variety over $K$.
 Indeed, if the $G_K$-module $W_{\ell}$ is  isomorphic to $V_{\ell}(B)$ for an abelian variety $B$ over $K$ then $\dim(B)=2$ and the theorem of Faltings implies that there is a {\sl nonzero} homomorphism of abelian varieties $B \to A$ over $K$, which is not the case, since the fourfold $A$ is simple. 
On the other hand, if $v \mid \ell$ then t $V_{\ell}(A)$ is a {\sl deRham representation} of $G_{K_v}$ with weights $0$ and $-1$, both of multiplicity $\dim(A)=4$. Since a subrepresentation of a deRham representation is also deRham, we conclude that 
$W_{\ell}$ is  deRham. It is also clear that $W_{\ell}$ has the same Hodge-Tate weights as 
$$V_{\ell}(A)=W_{\ell}\oplus W_{\ell}$$
but the multiplicities should be divided by $2$, i.e.,  the Hodge-Tate weights of $W_{\ell}$ are  $0$ and $-1$, both of multiplicity $2$.


We thus obtain:
\begin{thm}
\label{counterexample}
The system of representations $\{ W_{\ell} \}_{\ell \ne p}$ constructed above does not come globally from
an abelian variety defined over the field $K$ but for all $v \nmid \ell$ the representation $W_{\ell}$ locally 
comes from an abelian variety $B^v/K_v$. In particular,  $\{ W_{\ell} \}_{\ell \ne p}$ is a weakly compatible system of $4$-dimensional $\ell$-adic representations of $G_K$.

If $v \mid \ell$ then $W_{\ell}$ is locally a deRham representation with Hodge-Tate weights $0$ and $-1$, both of multiplicity $2$.
\end{thm}

{\bf Remark}
 By a theorem of Faltings \cite{Faltings}, the $G_K$-module $V_{\ell}(A)$ is semisimple and therefore its submodule
$W_{\ell}$ is also semisimple. On the other hand,  we know that
 the centralizer $$\End_{G_K}(W_{\ell})=k\otimes_{\Q}\Q_{\ell} \ne \Q_{\ell};$$ in particular, none of $W_{\ell}$ is absolutely irreducible.  In what follows we construct an example of a weakly compatible system (for all $\ell \ne p$) of absolutely irreducible deRham representations that does not come globally from an abelian variety over a number field.  However, we do not know whether it comes {\sl locally} from abelian varieties.

\vskip .5cm

Let $p$ be a prime and $H$ be a {\sl definite} quaternion algebra over $\Q$ that is unramified exactly at $p$ and $\infty$.   In particular, for each prime $\ell \ne p$ the $\Ql$-algebra
$$H\otimes_{\Q}\Ql\cong \MM_2(\Ql).$$
Let $g\ge 4$ be an even  integer.  According to Shimura (\cite{Shimura}, see also the case of Type III($e_0=1$) with $m=g/2$ in \cite[Table 8.1 on p. 498]{OortEndo} and \cite[Table on p. 23]{OZ}) there exists a complex $g$-dimensional abelian variety $X$, whose endomorphism algebra $\End^0(X)$ is isomorphic to $H$.   The same arguments as above (related to $D$)
 prove that there exists a $g$-dimensional  abelian variety $B$ over a certain number field $K$ such that all endomorphisms of $B$ are defined over $K$ and $\End^0(B)\cong H$. In particular, $B$ is absolutely simple.
By the theorem of Faltings,  if $\ell$ is a prime then the $G_K$-module $V_{\ell}(B)$ is semisimple and
$$\End_{G_K}(V_{\ell}(B))=H\otimes_{\Q}\Ql.$$
In particular, if $\ell \ne p$ then $\End_{G_K}(V_{\ell}(B))\cong \MM_2(\Q_{\ell})$ and therefore there are two isomorphic $\Q_{\ell}[G_K]$-submodules $U_{1,\ell}(B)$ and  $U_{2,\ell}(B)$ in $V_{\ell}(B)$ such that
$$V_{\ell}(B)= U_{1,\ell}(B)\oplus  U_{2,\ell}(B)\cong
U_{1,\ell}(B)\oplus U_{1,\ell}(B)\cong U_{2,\ell}(B)\oplus U_{2,\ell}(B)
.$$
If we denote by  $U_{\ell}$ the $\Q_{\ell}[G_K]$-module $U_{1,\ell}(B)$ then $\dim_{\Ql}(U_{\ell})=g$ and
we get an isomorphism of $\Q_{\ell}[G_K]$-modules
$$V_{\ell}(B)\cong U_{\ell}\oplus U_{\ell}.$$
Clearly, the submodule $U_{\ell}$ is semisimple and
$$\MM_2(\Q_{\ell})=H\otimes_{\Q}\Ql=\End_{G_K}(V_{\ell}(B))=\MM_2(\End_{G_K}(U_{\ell})).$$
This implies that  $\End_{G_K}(U_{\ell})=\Q_{\ell}$, i.e., the $\ell$-adic (sub)representation 
$$G_K \to \Aut_{\Ql}(U_{\ell})\cong \mr{GL}_g(\Ql)$$ 
is {\sl absolutely irreducible}.  Clearly, for each $\sigma \in G_K$  its characteristic polynomial with respect to the action on $V_{\ell}(B)$ is the square of its characteristic
 polynomial with respect to the action on $U_{\ell}$. This implies that if $v$ is an  nonarchimedean place $v$ of $K$ where $B$ has good reduction then for all primes $\ell \ne p$ such that $v \nmid \ell$ the characteristic polynomial of the frobenius element at $v$ with respect to its action on $U_{\ell}$ has rational coefficients and does not depend on $\ell$. In other words, $U_{\ell}$ is a weakly compatible system of (absolutely irreducible) $\ell$-adic representations.   As above, locally for each $v\mid \ell$ the $G_{K_v}$-module $V_{\ell}(B)$ is deRham with Hodge weights $0$ and $-1$ with weights $g$, which implies that $U_{\ell}$ is also deRham with the same Hodge-Tate weights, whose multiplicities are $g/2$.

\begin{thm}
\label{counterexample2}
The weakly compatible system of $g$-dimensional absolutely irreducible representations $\{ U_{\ell}\}_{\ell \ne p}$ constructed above does not come globally from
an  abelian variety defined over the field $K$.

If $v \mid \ell$ then $U_{\ell}$ is locally a deRham representation with Hodge-Tate weights $0$ and $-1$, both of multiplicity $g/2$.
\end{thm}

\begin{proof}
 We claim that none of $U_{\ell}$ comes out from an abelian variety over $K$. Indeed, if there is an abelian variety $C$ over $K$ such that the $G_K$-modules $V_{\ell}(C)$ and $U_{\ell}$ are isomorphic then $\dim(C)=g/2$ and the theorem of Faltings implies the existence of a {\sl nonzero} homomorphism $C \to B$, which contradicts the simplicity of $g$-dimensional $B$. 
\end{proof}

\section{Almost isomorphic abelian varieties}
\label{twists}

Throughout this section, $K$ is a field.  $A$ and $B$ are abelian varieties of positive dimension over $K$.  Recall that $\End^0(A)=\End(A)\otimes\Q$.
If $\ell$ is a positive integer then we write $\Z_{(\ell)}$ for the subring in $\Q$ that consists of all the rational numbers, whose denominators are prime to $\ell$. We have
$$\Z \subset \Z_{(\ell)}=\Z_{\ell}\bigcap \Q\subset \Z_{\ell}.$$
(Here the intersection is taken in $\Q_{\ell}$.)
In addition, if $m$ is a positive integer that is prime to $\ell$
then
$$\Z \subset \Z[1/m]\subset \Z_{(\ell)}\subset \Q.$$
The intersection of all $\Z_{(\ell)}$'s (in $\Q$) coincides with $\Z$.

 In this section we discuss the structure of the right
$\End(A)$-module $\Hom(A,B)$ when the $\Z_{\ell}$-Tate modules of
$A$ and $B$ are isomorphic as Galois modules for all  $\ell$ and $K$
is finitely generated over $\Q$. If $\ell \ne \mathrm{char}(K)$ and
$X$ is an abelian variety over $K$ then we write $X[\ell]$ for the
kernel of multiplication by $\ell$ in $X(\bar{K})$. It is well known
that
 $X[\ell]$ is a finite $G_K$-submodule in $X(\bar{K})$ of order $\ell^{2\dim(X)}$ and there is a natural isomomorphism of $G_K$-modules
$X[\ell]\cong T_{\ell}(X)/\ell T_{\ell}(X)$.
\begin{lem}
\label{rankOneQ}
Let $A$ and $B$ be abelian varieties of positive dimension over $K$.

\begin{enumerate}
\item[(a)]
If $A$ and $B$ are isogenous  over $K$ then the right $\End(A)\otimes\Q$-module 
\newline
$\Hom(A,B)\otimes\Q$ is free of rank $1$. In addition, one may choose as a generator of $\Hom(A,B)\otimes\Q$ any isogeny $\phi:A \to B$.
\item[(b)]
The following conditions are equivalent.

\begin{itemize}
\item[(i)]
The right $\End(A)\otimes\Q$-module $\Hom(A,B)\otimes\Q$ is free of rank $1$.
\item[(ii)]
$\dim(A) \le \dim(B)$ and there exists a $\dim(A)$-dimensional abelian $K$-subvariety $B_0\subset B$ such that $A$ and $B_0$ are isogenous over $K$ and
$$\Hom(A,B)=\Hom(A,B_0).$$
In particular, the image of every $K$-homomorphism  of abelian varieties
$A \to B$ lies in $B_0$.
\end{itemize}
\item[(c)]
If the equivalent conditions (i) and (ii) hold and $\dim(B)\le \dim(A)$ then $\dim(A)=\dim(B), B=B_0$, and $A$ and $B$ are isogenous over $K$.

\end{enumerate}
\end{lem}

\begin{proof}
(a) is obvious.

Suppose (bii) is true. Let us pick an {\sl isogeny} $\phi: A \to B_0$. It follows that $\Hom(A,B_0)\otimes\Q=\phi \End^0(A)$ is a free right $\End^0(A)$-module of rank $1$ generated by $\phi$. Now (bi) follows from the equality
$$\Hom(A,B)\otimes\Q=\Hom(A,B_0)\otimes\Q.$$

Suppose that (bi) is true. We may choose a homomorphism of abelian varieties
$\phi: A \to B$  as a generator (basis) of the free right $\End(A)\otimes\Q$-module $\Hom(A,B)\otimes\Q$. In other words,  for every homomorphism of abelian varieties $\psi:A \to B$ there are $u \in \End(A)$ and a {\sl nonzero} integer $n$ such that
$n\psi=\phi u$.
In addition, for each {\sl nonzero} $u \in \End(A)$ the composition $\phi u$ is a {\sl nonzero} element of $\Hom(A,B)$.
Clearly, $B_0:=\phi(A)\subset B$ is an abelian $K$-subvariety of $B$
with $\dim(B_0) \le  \dim(A)$.
We have
$$n\psi(A) =\phi u(A) \subset \psi(A)\subset B_0.$$
It follows that the identity component of $\psi(A)$ lies in $B_0$.
Since $\psi(A)$ is a (connected) abelian $K$-subvariety of $B$, we have
$\psi(A)\subset B_0$. This proves that
$\Hom(A,B)=\Hom(A,B_0)$.
On the other hand,
 if $\dim(B_0)=\dim(A)$ then
$\phi: A \to B_0$ is an {\sl isogeny} and
we get (bii) under our additional assumption.
If $\dim(B_0)<\dim(A)$ then   $\ker(\phi)$ has positive dimension that is strictly less than $\dim(A)$. By the Poincar\'e complete reducibility theorem there is an endomorphism $u_0 \in \End(A)$ such that the image $u_0(A)$ coincides with the identity component of $\ker(\phi)$; in particular,
$u_0 \ne 0, \ u_0(A) \subset \ker(\phi)$.
 This implies that
$\phi u_0=0 $ in $\Hom(A,B)$
and we get a contradiction, which proves (bii).

(c) follows readily from (bii).
\end{proof}


\begin{lem}
\label{homoEnd}
Suppose that $A,B,C$ are abelian varieties over $K$ of positive dimension that are mutually isogenous over $K$. We view $\Hom(A,B)\otimes\Q$ and $\Hom(A,C)\otimes\Q$ as right $\End^0(A)=\End(A)\otimes\Q$-modules. Then the natural map
$$m_{B,C}:\Hom(B,C)\otimes\Q \to \Hom_{\End^0(A)}(\Hom(A,B)\otimes\Q,\Hom(A,C)\otimes\Q)$$
that associates to $\tau: B \to C$ a homomorphism of right $\End(A)\otimes\Q$-modules
$$m_{B,C}(\tau):\Hom(A,B)\otimes\Q\to\Hom(A,C)\otimes\Q, \psi \mapsto \tau\psi$$
is a group isomorphism.
\end{lem}

\begin{proof}
Clearly, $m_{B,C}$ is injective. In order to check the surjectiveness, 
notice that the statement is clearly {\sl invariant  by isogeny}, so we can
assume that $B=A$ and $C=A$, in which case it is obvious.
\end{proof}

Now till the end of this section we assume that $K$ is a field of characteristic zero that is finitely generated over $\Q$, and $A$ and $B$ are abelian varieties of positive dimension over $K$.  By a theorem of Faltings \cite{Faltings,Faltings2},
$$\Hom_{G_K}(T_{\ell}(A), T_{\ell}(B))=\Hom(A,B)\otimes \Z_{\ell}. \eqno{(*)}$$

\begin{lem}
\label{ellIso}
Let $\ell$ be a prime.
Then the following conditions are equivalent.
\begin{itemize}
\item[(i)]
There is an isogeny $\phi_{\ell}:A \to B$, whose degree is prime to $\ell$.
\item[(ii)]
The Tate modules $T_{\ell}(A)$ and $T_{\ell}(B)$ are isomorphic as $\Z_{\ell}[G_K]$-Galois modules.
\end{itemize}
If the equivalent conditions (i) and (ii) hold then  the right $\End(A)\otimes \Z_{(\ell)}$-module $\Hom(A,B)\otimes \Z_{(\ell)}$ is free of rank $1$ and the right $\End(A)\otimes \Z_{\ell}$-module $\Hom(A,B)\otimes \Z_{\ell}$ is free of rank $1$
\end{lem}

\begin{proof}
(i) implies (ii). Indeed,  let $\phi_{\ell}: A \to B$ be an isogeny such that its degree $d:=\deg(\phi_{\ell})$ is prime to $\ell$. Then there exists an isogeny
$\varphi_{\ell}: B \to A$ such that $\phi_{\ell} \varphi_{\ell}$ is multiplication by $d$  in $B$ and $\varphi_{\ell} \phi_{\ell}$ is multiplication by $d$  in $A$. This implies that  $\phi_{\ell}$ induces an $G_K$-equivariant isomorphism of the $\Z_{\ell}$-Tate modules of $A$ and $B$.

Suppose that (ii) holds.  Since the rank of the free $\Z_{\ell}$-module $T_{\ell}(A)$ (resp. $T_{\ell}(B)$) is $2\dim(A)$ (resp. $2\dim(B)$), we conclude that
$2\dim(A)=2\dim(B)$, i.e.
$\dim(A)=\dim(B)$.
 By the theorem of Faltings (*), there is an isomorphism of the $\Z_{\ell}$-Tate  modules of $A$ and $B$ that lies in $\Hom(A,B)\otimes \Z_{\ell}$.  Since $\Hom(A,B)$ is dense in $\Hom(A,B)\otimes \Z_{\ell}$ in the $\ell$-adic topology, and the set of isomorphisms $T_{\ell}(A)\cong T_{\ell}(B)$ is open in $\Hom(A,B)\otimes \Z_{\ell}$, there  is $\phi_{\ell}\in \Hom(A,B)$ that induces an isomorphism $T_{\ell}(A)\cong T_{\ell}(B)$. Clearly, $\ker(\phi_{\ell})$ does not contain points of order $\ell$ and therefore is finite. This implies that $\phi_{\ell}$ is an isogeny, whose degree is prime to $\ell$. This proves (i).

In order to prove the last assertion of Lemma \ref{ellIso}, one has only to observe that
$\phi_{\ell} \in \Hom(A,B)\subset \Hom(A,B)\otimes \Z_{(\ell)}\subset
\Hom(A,B)\otimes \Z_{\ell}$
is a generator of  the (obviously) free right $\Z_{(\ell)}$-module $\Hom(A,B)\otimes \Z_{(\ell)}$ and of the free right $\Z_{\ell}$-module $\Hom(A,B)\otimes \Z_{\ell}$.
\end{proof}

We say that $A$ and $B$ are {\sl almost isomorphic} if for {\sl all primes} $\ell$ the equivalent conditions (i) and (ii) of Lemma \ref{ellIso} hold. Clearly, if $A$ and $B$ are isomorphic over $K$ then they are almost isomorphic. It is also clear that if $A$ and $B$ are almost isomorphic then they are isogenous over $K$. Obviously, the property of being almost isomorphic is an equivalence relation on the set of (nonzero) abelian varieties over $K$.

\begin{cor}
\label{isomTateFree}
Suppose that $A$ and $B$ are almost isomorphic. Then $A$ and $B$ are isomorphic over $K$ if and only if $\Hom(A,B)$ is a free $\End(A)$-modules of rank $1$.
In particular, if $\End(A)$ is a principal ideal domain
 (for example, 
$\End(A)=\Z$) then every abelian variety over $K$, which is almost
isomorphic to $A$, is actually isomorphic to $A$.
\end{cor}

\begin{proof}
Suppose $\Hom(A,B)$ is a free $\End(A)$-module, i.e., there is a homomorphism of abelian varieties $\phi: A \to B$ such that $\Hom(A,B)=\phi \End(A)$.  We know that for any prime $\ell$ there is an isogeny $\phi_{\ell}: A \to B$ of degree prime to $\ell$. (In particular, $\dim(A)=\dim(B)$.) Therefore there is $u_{\ell} \in \End(A)$ with $\phi_{\ell}=\phi u_{\ell}$.
In particular, $\phi_{\ell}(A)\subset \phi(A)$ and $\deg(\phi_{\ell})$ is divisible by $\deg(\phi)$.  Since $\phi_{\ell}(A)=B$ and $\deg(\phi_{\ell})$ is prime to $\ell$, we conclude that $\phi(A)=B$ (i.e., $\phi$ is an isogeny) and $\deg(\phi)$ is prime to $\ell$. Since the latter is true for all primes $\ell$, we conclude that $\deg(\phi)=1$, i.e., $\phi$ is an isomorphism.

Conversely, if $A\cong B$ then $\Hom(A,B)$ is obviously a free $\End(A)$-module generated by an isomorphism between $A$ and $B$.

The last assertion of Corollary follows from the well-known fact
that every finitely generated module without torsion over a
principal ideal domain is free.
\end{proof}

The next statement is a generalization of Corollary \ref{isomTateFree}.

\begin{cor}
\label{BisoC}
Suppose that $A,B,C$ are abelian varieties of positive dimension over $K$ that are almost isomorphic to each other.

Then $B$ and $C$ are isomorphic over $K$ if and only if the right $\End(A)$-modules $\Hom(A,B)$ and $\Hom(A,C)$ are isomorphic.
\end{cor}

\begin{proof}
We know that all $A,B,C$ are mutually isogenous over $K$. Let us choose an isogeny $\phi: B \to C$. We are given an isomorphism
$\delta: \Hom(A,B)\cong \Hom(A,C)$
of right $\End(A)$-modules
that obviously extends by $\Q$-linearity to the isomorphism
$\Hom(A,B)\otimes \Q \to \Hom(A,C)\otimes \Q$ of right $\End(A)\otimes\Q$-modules, which we continue to denote by $\delta$.  By Lemma \ref{homoEnd}, there exists $\tau_0 \in \Hom(B,C)\otimes \Q$ such that $\delta=m_{B,C}(\tau_0)$, i.e.,
$$\delta(\psi)=\tau_0\psi \ \forall \psi \in \Hom(A,B)\otimes \Q.$$
There exists a positive integer $n$ such that $\tau=n\tau_0\in \Hom(B,C)$ and $\tau$ is {\sl not} divisible in $\Hom(B,C)$. This implies that
$$n\cdot \Hom(A,C)=n \delta (\Hom(A,B))=n \tau_0\Hom(A,B)=\tau\Hom(A,B).$$
Since $B$ and $C$ are almost isomorphic,
 for each $\ell$ there is an isogeny $\phi_{\ell}: B \to C$ of degree prime to $\ell$. Since $n \phi_{\ell} \in \tau\Hom(A,B)$, we conclude that $\tau$ is an isogeny and  $\deg(\tau)$ is prime to $\ell$ if $\ell$ does {\sl not} divide $n$.  We need to prove that $\tau$ is an isomorphism. Suppose it is not, then there is a prime $\ell$ that divides $\deg(\tau)$ and therefore divides $n$. We need to arrive to a contradiction. Since $A$ and $B$ are almost isomorphic,  there is an isogeny $\psi_{\ell}:A \to B$ of degree prime to $\ell$. We have
$\tau \psi_{\ell}\in n\cdot \Hom(A,C) \subset \ell\cdot \Hom(A,C)$.
This implies that $\tau$ kills {\sl all} points of order $\ell$ on $B$ and therefore is divisible by $\ell$ in $\Hom(B,C)$, which is not the case. This gives us the desired contradiction.
\end{proof}

\begin{rem}
\label{centerAB} Let ${\mathcal Z}(A)$ (resp. ${\mathcal Z}(B)$) be
the  the center of $\End(A)$ (resp. $\End(B)$). Then ${\mathcal
Z}(A)_{\Q}:={\mathcal Z}(A)\otimes\Q$ (resp. ${\mathcal
Z}(B)_{\Q}:={\mathcal Z}(B)\otimes\Q$) is the center of
$\End(A)\otimes\Q$ (resp. $\End(B)\otimes\Q$) and for all primes
$\ell$ the $\Z_{(\ell)}$-subalgebra
$${\mathcal Z}(A)_{(\ell)}:={\mathcal
Z}(A)\otimes\Z_{(\ell)}\subset {\mathcal Z}(A)_{\Q}\subset
\End(A)\otimes\Q$$ (resp. the $\Z_{(\ell)}$-subalgebra
$${\mathcal Z}(B)_{(\ell)}:={\mathcal
Z}(B)\otimes\Z_{(\ell)}\subset {\mathcal Z}(B)_{\Q}\subset
\End(B)\otimes\Q)$$ is the center of $\End(A)\otimes\Z_{(\ell)}$
(resp. of $\End(B)\otimes\Z_{(\ell)}$).  Every $K$-isogeny $\phi: A \to
B$ gives rise to an isomorphism of $\Q$-algebras
$$i_{\phi}:\End(A)\otimes\Q \cong \End(B)\otimes\Q, \ u \mapsto  \phi u
\phi^{-1},$$ such that $i_{\phi}({\mathcal Z}(A)_{\Q})={\mathcal
Z}(B)_{\Q}$ and the restriction
$i_{\mathcal Z}: {\mathcal Z}(A)_{\Q} \cong {\mathcal Z}(B)_{\Q}$
of $i_{\phi}$ to the center(s) does {\sl not} depend on a choice of
$\phi$ \cite{ZarhinLuminy}. If $\phi_{\ell}: A \to B$ is a $K$-isogeny
of degree prime to $\ell$ then
$i_{\phi_{\ell}}(\End(A)\otimes\Z_{(\ell)})=\End(B)\otimes\Z_{(\ell)}$
and therefore
$i_{\mathcal Z}({\mathcal Z}(A)_{(\ell)}) )= {\mathcal Z}(B)_{(\ell)}$.
This implies that if $A$ and $B$ are {\sl almost isomorphic} then
$i_{\mathcal Z}({\mathcal Z}(A))$ coincides with ${\mathcal Z}(B)$ and therefore
$i_{\mathcal Z}$ defines a canonical isomorphism of commutative
rings ${\mathcal Z}(A)\cong {\mathcal Z}(B)$. In particular, if
$\End(A)$ is commutative then $\End(B)$ is also commutative (because
$\End(A)\otimes\Q$ and $\End(B)\otimes\Q$ are isomorphic) and there
is a canonical ring isomorphisms $\End(A)\cong\End(B)$.
\end{rem}

Until the end of this section,  $\Lambda$ is  a ring with $1$ that, viewed as an additive group, is a free $\Z$-module of finite positive rank. In addition, we assume that the finite-dimensional $\Q$-algebra $\Lambda_{\Q}:=\Lambda\otimes\Q$ is {\sl semisimple}. We write
$\Lambda_{\ell}$ (resp. $\Lambda_{(\ell)}$) for the $\Z_{\ell}$-algebra
$\Lambda\otimes \Z_{\ell}$ (resp. for the $\Z_{(\ell)}$-algebra
$\Lambda\otimes \Z_{(\ell)}$).  We have 
$$\Lambda=\Lambda \otimes 1\subset \Lambda_{(\ell)}\subset \Lambda_{\Q}\subset \Lambda\otimes\Q_{\ell},$$
$$\Lambda\subset \Lambda_{(\ell)}\subset
\Lambda_{\ell}\subset \Lambda\otimes\Q_{\ell}.$$
In addition, the intersection of $\Lambda_{\ell}$ and $\Lambda_{\Q}$
(in $\Lambda\otimes\Q_{\ell}$) coincides with $\Lambda_{(\ell)}$.

 Let $M$  be  an {\sl arbitrary} free commutative group of finite positive rank that is provided with a structure of a right $\Lambda$-module. We write $M_{\Q}$ for the right  $\Lambda_{\Q}$-module $M\otimes \Q$,
$M_{\ell}$ for the right  $\Lambda_{\ell}$-module $M\otimes \Z_{\ell}$ and $M_{(\ell)}$ for the right  $\Lambda_{(\ell)}$-module $M\otimes \Z_{(\ell)}$. We have 
$$M=M \otimes 1\subset M_{(\ell)}\subset M_{\Q}\subset M\otimes\Q_{\ell},$$
$$M\subset M_{(\ell)}\subset
M_{\ell}\subset M\otimes\Q_{\ell}.$$
In addition, the intersection of $M_{\ell}$ and $M_{\Q}$
(in $M\otimes\Q_{\ell}$) coincides with $M_{(\ell)}$.

{\bf Definition}. We say that $M$ is a {\sl locally free right $\Lambda$-module of rank $1$} if for all primes $\ell$ the right $\Lambda_{\ell}$-module $M_{\ell}$ is free of rank $1$. (See \cite{Fro}.)

\begin{thm}
\label{freeRankOne}
Let $M$ be a locally free right $\Lambda$-module of rank $1$. Then it enjoys the following properties.
\begin{itemize}

\item[(i)]
$M$ is a projective $\Lambda$-module. More precisely, $M$ is isomorphic to a direct summand 
of a free right $\Lambda$-module of rank $2$.
\item[(ii)]
The right $\Lambda_\Q$-module $M_{\Q}$ is free of rank $1$. 
\item[(iii)]
The right $\Lambda_{(\ell)}$-module $M_{(\ell)}$ is free of rank $1$ for all primes $\ell$. 
\end{itemize}
\end{thm}

\begin{proof}
Let $J(\Lambda_{\Q})$ be the (multiplicative) {\sl idele group}  of $\Lambda_{\Q}$, i.e., the group of invertible elements of the {\sl adele ring} of $\Lambda_{\Q}$ \cite[p. 114]{Fro}. (In the notation of \cite[Sect. 2]{Fro}, 
$\mathfrak{o}=\Z,  \ K=\Q, A=\Lambda_{\Q}, \ \mathfrak{U}=\Lambda$.) 
To each $\alpha \in J(\Lambda_{\Q})$ corresponds a certain right $\Lambda$-submodule $\alpha \Lambda\subset \Lambda_{\Q}$ that is 
a locally free $\Lambda$-module of rank $1$ and a $\Z$-lattice of maximal rank in the
$\Q$-vector space $\Lambda_{\Q}$, i.e., the natural homomorphism of $\Q$-vector spaces
$\alpha \Lambda\otimes \Q\to \Lambda_{\Q}$ is an isomorphism \cite[p. 114]{Fro}.  This implies that $(\alpha \Lambda)_{\Q}$ is a free $\Lambda_{\Q}$-module of rank $1$.
In addition, the direct sum $\alpha \Lambda\oplus \alpha^{-1} \Lambda$ is a free right $\Lambda$-module of rank 2 \cite[Th. 1  on pp. 114--115]{Fro}.
This implies that $\alpha \Lambda$ is isomorphic to a direct summand of a rank $2$ free module; in particular, it is projective. 
By the same Theorem 1 of \cite{Fro}, every right locally free $\Lambda$-module $M$ of rank $1$ is isomorphic to $\alpha \Lambda$ for a suitable $\alpha$.
This proves (i) and (ii).

Let $f_0$ be a generator of the free $\Lambda_\Q$-module $M_{\Q}$ of rank $1$. Multiplying $f_0$ by a sufficiently divisible
positive integer, we may and will assume that
$f_0 \in M=M\otimes 1\subset M_{\Q}$.
Clearly,  the right $\Lambda\otimes\Q_{\ell}$-module 
$$M\otimes\Q_{\ell}=M_{\Q}\otimes_{\Q}\Ql= M_{\ell}\otimes_{\Z_{\ell}}\Ql$$
 is free of rank $1$ for all primes $\ell$ and
 $f_0$ is also a generator of $M\otimes\Q_{\ell}$. It is also clear that every generator $f_{\ell}$ of the 
$\Lambda_{\ell}$-module $M_{\ell}$  is a generator of the $\Lambda\otimes\Q_{\ell}$-module 
$M\otimes\Q_{\ell}$.  We claim that there is a generator $f_{\ell}$ that lies in $M$.
 Indeed, with respect to the  $\ell$-adic topology, the subset
$$M=M\otimes 1 \subset M\otimes\Z_{\ell}=M_{\ell}$$ is dense in 
$M_{\ell}$  while the set of generators of the free  $\Lambda_{\ell}$-module $M_{\ell}$ is open, because
the group of units $(\Lambda_{\ell})^{*}$ is open in $\Lambda_{\ell}$.
 This implies that there exists a (nonzero) generator 
$f_{\ell} \in M\subset M_{\ell}$
of the $\Lambda_{\ell}$-module $M_{\ell}$. 
Recall that $f_{\ell}$ is also a generator of the free $\Lambda\otimes\Q_{\ell}$-module $M\otimes\Q_{\ell}$.
This implies that there exists 
$\mu_0 \in (\Lambda\otimes\Ql)^{*}$
such that
$f_{\ell}=f_0 \mu_0 \in M\otimes\Q_{\ell}$.
On the other hand, since $f_{\ell}$ lies in the free rank $1$ $\Lambda_{\Q}$-module $M_{\Q}=f_0\Lambda_{\Q}$, we have
$\mu_0\in \Lambda_{\Q}$.
This implies that $\mu_{0}$ is {\sl not} a zero divisor in the finite-dimensional $\Q$-algebra $\Lambda_{\Q}$ (because it is invertible in $\Lambda\otimes\Ql$)
  and therefore lies in $\Lambda_{\Q}^{*}$.
It follows that $f_{\ell}$ is also a generator of the free $\Lambda_{\Q}$-module $M_{\Q}$ of rank $1$.

We want to prove that
$M_{(\ell)}=f_{\ell}[\Lambda\otimes\Z_{(\ell)}]$.
(This would prove that $M_{(\ell)}$ is a free right $\Lambda_{(\ell)}$-module  of rank $1$ with the generator $f_{\ell}$.)
For each $x \in M_{(\ell)}$ there exists a unique $\lambda \in \Lambda_{\ell}$ with $x= f\lambda$.
We need to prove that $\lambda \in \Lambda_{(\ell)}$. Notice that
$x \in M_{(\ell)}\subset M_{\Q}$.
Since $f_{\ell}$ is a generator of the free $\Lambda_{\Q}$-module $M_{\Q}$, there exists exactly one
$\mu_0 \in \Lambda_{\Q}$ such that
$x=f \mu_0$. 
We get the  equalities  $f  \mu_0=x=f \mu$ in $M\otimes\Q_{\ell}$.

Since $f_{\ell}$ is a generator of the  free $\Lambda\otimes\Ql$-module $M\otimes\Ql$, we get
$\mu=\mu_0$.
Since $\Lambda_{(\ell)}$  coincides with intersection of $\Lambda_{\ell}$ and $\Lambda_{\Q}$ in $\Lambda\otimes\Ql$,
we conclude that
$\mu=\mu_0 \in \Lambda_{(\ell)}$
and therefore
$x \in f [\Lambda\otimes\Z_{(\ell)}]$.
This implies that $M_{(\ell)}$ is a free right $\Lambda_{(\ell)}$ module of rank $1$, which proves (iii).
\end{proof}

\begin{cor}
\label{rankOneC}
Let $M$  be  a free commutative group of finite positive rank that is provided with a structure of a right $\Lambda$-module.  
Then $M$ is a locally free $\Lambda$-module of rank $1$ if and only if the right $\Lambda_{(\ell)}$-module $M_{(\ell)}$ is free of rank $1$ for all primes $\ell$.
\end{cor}

\begin{proof}
Clearly, if $M_{(\ell)}$ is a free right $\Lambda_{(\ell)}$-module of rank $1$ then the right $\Lambda_{\ell}$-module $M_{\ell}$ is free of rank $1$.  The converse follows from Theorem \ref{freeRankOne}(iii).
\end{proof}

\begin{rem}\label{examplesM}
Suppose that $\Lambda$ is an {\sl order} in a number field $E$, i.e., $\Lambda$ is a finitely generated over $\Z$ a subring (with $1$) of $E$ such that
$\Lambda_{\Q}=E$.  Let $M$ be a $\Lambda$-module in $E$, i.e., a free commutative additive  (sub)group of finite rank in  $E$ such that $M \cdot\Lambda=M$.
In particular, $M_{\Q}=E$ is a free $E=\Lambda_{\Q}$-module of rank $1$.
\begin{itemize}
\item[(i)]
If $\Lambda$ is the ring of all integers in $E$ then it is a
Dedekind ring and each of its {\sl localizations} $\Lambda_{(\ell)}$ is
a Dedekind ring with finitely many maximal ideals and therefore is a
{\sl principal ideal domain} \cite[Ch. III, Prop. 2.12 on
p.93]{Lorenzini}. This implies that $M_{(\ell)}$ is a free
$\Lambda_{(\ell)}$-module, whose rank is obviously $1$.  By Corollary \ref{rankOneC},
 $M$ is locally free of rank $1$.

\item[(ii)]
Suppose that $E$ is a quadratic field. We don't impose any restrictions on $\Lambda$ but instead assume that $\End_{\Lambda}(M)=\Lambda$.
Then it is known \cite[Lemma 2 on p. 55]{BSh} that for each  prime $\ell$ there is a nonzero ideal $\mathfrak{J}\subset \Lambda$ such that the order of the finite quotient $\Lambda/\mathfrak{J}$ is prime to $\ell$ and  the $\Lambda$-modules $M$ and $\mathfrak{J}$ are isomorphic. This implies that the $\Lambda_{(\ell)}$-module $J_{(\ell)}=\Lambda_{(\ell)}$ is free and therefore the $\Lambda_{(\ell)}$-module $M_{(\ell)}$ is also free and its rank is obviously $1$. 
 By Corollary \ref{rankOneC},
 $M$ is locally free of rank $1$.
\end{itemize}
\end{rem}

Now we are going to use Theorem \ref{freeRankOne}, in order to
construct abelian varieties $A\otimes M$ over $K$ that are {\sl
almost isomorphic} to a given $A$.
Notice that our $A\otimes M$ are a rather special {\sl naive} case
of powerful {\sl Serre's tensor construction} (\cite[Sect.
7]{Conrad}, \cite[Sect. 1.7.4]{CCFO}).

Suppose we are given a a free commutative group $M$ of finite
(positive) rank that is provided with a structure of a right locally free
$\Lambda=\End(A)$-module  of rank $1$.  Let $F_2$ be a free right $\Lambda$-module of rank $2$.
It follows from Theorem \ref{freeRankOne}(i) that
there is an   endomorphism
$\gamma: F_2 \to F_2$
of the right $\Lambda$-module $F_2$ such that $\gamma^2=\gamma$ and
 whose image
$M^{\prime}=\gamma(F_2)$ is isomorphic to $M$. Notice that
$\End_{\Lambda}(F_2)$ is the matrix algebra $\MM_2(\Lambda)$ of size
$2$ over $\Lambda$.  So, the idempotent
$$\gamma \in \End_{\Lambda}(F_2)=\MM_2(\Lambda)=\MM_2(\End(A))=\End(A^2)$$
where $A^2=A\times A$.
Let us define the $K$-abelian (sub)variety
$$B=A\otimes M:=\gamma(A^2)\subset A^2.$$
Clearly, $B$ is a direct factor of $A^2$. More precisely, if we consider the
\newline
 $K$-abelian (sub)variety
$C=(1-\gamma)(A^2)\subset A^2$
then the natural homomorphism 
$B \times C \to A^2, \ (x,y) \mapsto x+y$
of abelian varieties over $K$
is an isomorphism, i.e., $A^2=B \times C$. This implies that the right $\End(A)$-module $\Hom(A,B)$ coincides with
$$\gamma \Hom(A,A^2) \subset \Hom(A,A^2)=\End(A)\oplus \End(A)=F_2$$
and therefore the right $\End(A)$-module $\Hom(A,B)$ is canonically isomorphic to $\gamma(F_2)=M^{\prime} \cong M$.
It also follows that for every prime $\ell$
$$\gamma(A^2[\ell])=B[\ell] . \eqno{(**)}$$

\begin{thm}
\label{twistM}
Let us consider the abelian variety $B=A\otimes M$ over $K$. Then:
\begin{itemize}
\item[(i)]
$A$ and $B$ are isogenous over $K$.
\item[(ii)]
The right $\End(A)$-module $\Hom(A,B)$ is isomorphic to $M$.
\item[(iii)]
$A$ and $B$ are almost isomorphic.
\end{itemize}
\end{thm}

\begin{proof}
We have already seen that $\Hom(A,B)\cong M$, which proves (ii).

Since the right $\End(A)\otimes\Q$-module $M\otimes \Q$ is free of rank $1$, the same is true for the  right $\End(A)\otimes\Q$-module $\Hom(A,B)$. By Lemma \ref{rankOneQ},
$\dim(A) \le \dim(B)$ and there exists a $\dim(A)$-dimensional abelian $K$-subvariety $B_0\subset B$ such that $A$ and $B_0$ are isogenous over $K$ and
$$\Hom(A,B)=\Hom(A,B_0).\eqno{(***)}$$
We claim that $B=B_0$. Indeed, if $B_0 \ne B$ then, by the Poincar\'e Complete Reducibility theorem \cite[Th. 6 on p. 28]{Lang}, there is an ``almost complimentary" abelian $K$-subvariety $B_1\subset B$ of positive dimension $\dim(B)-\dim(B_0)$ such that the intersection $B_0\bigcap B_1$ is finite and $B_0+B_1=B$. It follows from (***) that $\Hom(A,B_1)=\{0\}$.
However, $B_1\subset B \subset A^2$ is an abelian $K$-subvariety of $A^2$ and therefore there is a surjective homomorphism $A^2 \to B$ and therefore there exists a nonzero homomorphism $A\to B$. This is a contradiction, which proves that $B=B_0$, the right $\End(A)$-module $\Hom(A,B)$ is isomorphic to $M$, and $A$ and $B$ are isogenous over $K$.  In particular, $\dim(A)=\dim(B)$.  This proves (i).

Let $\ell$ be a prime.
Since $M\otimes\Z_{\ell}$ is a free right $\End(A)\otimes\Z_{\ell}$-module of rank $1$, $\Hom(A,B)\otimes\Z_{\ell}$ is a free right $\End(A)\otimes\Z_{\ell}$-module of rank $1$. Let us choose a generator
$\phi \in \Hom(A,B)$ of the module $\Hom(A,B)\otimes\Z_{\ell}$.  The {\sl surjection}
$\gamma:A^2\to B\subset A^2$
 is defined by a certain pair of homomorphisms $\phi_1, \phi_2:A\to B$, i.e.,
$$\gamma(x_1,x_2)=\phi_1(x_1)+\phi_2(x_2) \ \forall (x_1,x_2) \in A^2.$$
Since $\phi$ is the generator, there are $u_1,u_2 \in \End(A)\otimes\Z_{\ell}$ such that
$$\phi_1=\phi u_1, \ \phi_1=\phi u_1$$
in $\Hom(A,B)\otimes \Z_{\ell}$. It follows that
$$\gamma(A^2[\ell])=\phi_1 (A[\ell])+ \phi_2 (A[\ell])=
\phi u_1 (A[\ell])+ \phi u_2 (A[\ell])\subset \phi(A[\ell])\subset B[\ell].$$
By (**), $\gamma(A^2[\ell])=B[\ell]$. This implies that $\phi$ induces a surjective homomorphism $A[\ell] \to B[\ell]$. Since finite groups $A[\ell]$ and $B[\ell]$ have the same order, $\phi$ induces an isomorphism $A[\ell] \to B[\ell]$. This implies that $\ker(\phi)$ does not contain points of order $\ell$ and therefore is an {\sl isogeny} of degree prime to $\ell$.  This proves (iii).
\end{proof}

\begin{cor}
\label{isomTwist}
Suppose that for each $i=1,2$ we are given a commutative free group $M_i$ of finite positive rank provided with  the structure of a right locally free $\End(A)$-module of rank $1$.
Then abelian varieties $B_1=A\otimes M_1$ and $B_2=A\otimes M_2$ are isomorphic over $K$ if and only if the $\End(A)$-modules $M_1$ and $M_2$ are isomorphic.
\end{cor}

\begin{proof}
By Theorem \ref{twistM}(ii), the right $\End(A)$-module $\Hom(A,B_i)$ is isomorphic to $M_i$. Now the result follows from Theorem \ref{twistM}(iii)
combined with Corollary  \ref{BisoC}.
\end{proof}

\begin{cor}
\label{AtensorHomAB}
Let $A$ and $B$ be abelian varieties over $K$ of positive dimension. Suppose that
 the Galois modules $T_{\ell}(A)$ and $T_{\ell}(B)$ are isomorphic for all primes $\ell$. Then  abelian varieties $B$ and $C:=A\otimes\Hom(A,B)$ are isomorphic over $K$.
\end{cor}

\begin{proof}
By Theorem \ref{twistM}(ii), the right $\End(A)$-module $\Hom(A,C)$ is isomorphic to $\Hom(A,B)$.  Now  the result follows from Theorem \ref{twistM}(iii)
combined with Corollary  \ref{BisoC}.
\end{proof}

\begin{rem}
Suppose that $A$ is the product $A_1 \times A_2$ where $A_1$ and
$A_2$ are abelian varieties of positive dimension over $K$ with
$\Hom(A_1,A_2)=\{0\}$. Then $\End(A)=\End(A_1)\oplus \End(A_2)$.
Suppose that for each  $i=1,2$ we are given a commutative free group
$M_i$ of finite positive rank provided with the structure of a right locally free
$\End(A_i)$-module  of rank $1$.
 Then the direct sum $M=M_1\oplus M_2$ becomes a right locally free module of rank $1$ over the ring
$\End(A_1)\oplus \End(A_2)=\End(A)$.
 There is an obvious canonical
isomorphism between abelian varieties $A\otimes M$ and $(A_1\otimes
M_1)\times (A_2\otimes M_2)$ over $K$.

For example, suppose  $A_1$ is principally polarized, 
$\End(A_1)=\Z$ and  all $\bar{K}$-endomorphisms of $A$ are defined over $K$; in particular, $A$ is absolutely simple.
Let us take  $M_1=\Z$. Let $A_2$ be an elliptic curve such
that $\End(A_2)$ is the ring of integers in an imaginary quadratic
field with class number $>1$. Clearly, $\Hom(A_1,A_2)=\{0\}$. 
Actually, every $\bar{K}$-homomorphism between $A_1$ and $A_2$ is $0$.
Let
$M_2$ be a {\sl non-principal} ideal in $\End(A_2)$. Then elliptic
curves $A_2$ and $A_2\otimes M$ are almost isomorphic but are {\sl not
isomorphic} over $K$ and even over $\bar{K}$. This implies that $A\otimes M=A_1\times (A_2\otimes
M_2)$ is almost isomorphic over $K$ but  is {\sl not} isomorphic to $A=A_1\times
A_2$ over $\bar{K}$. On the other hand, both $A$ and $A\otimes M$ are principally
polarized, since $A_1$ is principally polarized while both $A_2$ and
$A_2\otimes M_2$ are elliptic curves.
\end{rem}

\section{Moduli of curves}

The moduli space of smooth projective curves of genus $g$ is denoted by $\M_g$. It is also
an orbifold and we will consider its fundamental group as such. For definitions see \cite{Hain}. It is defined
over $\Q$ and thus we can consider it over an arbitrary number field $K$.  As per our earlier
conventions, $\bar{\M}_g$ is the base change of $\M_g$ to an algebraic closure of $\Q$ and not a compactification.

Let $X$ be a curve of genus $g$ defined over $K$. There is a map
(an arithmetic analogue of the Dehn-Nielsen-Baer theorem, see \cite{MT},
in particular, lemma 2.1)
$\rho: \pi_1(\M_g) \to \Out(\pi_1(\bar{X}))$. This follows by considering the
universal curve $\Cu_g$ of genus $g$ together with the map $\Cu_g \to \M_g$,
so $X$ can be viewed as a fiber of this map. This gives rise to the fibration
exact sequence

$$1 \to \pi_1(\bar{X}) \to \pi_1(\Cu_g) \to \pi_1(\M_g) \to 1 $$

\noindent
and the action of $\pi_1(\Cu_g)$ on $\pi_1(\bar{X})$ gives $\rho$. Now,
$X$, viewed as a point on $\M_g(K)$, gives a map
$\sigma_{\M_g/K}(X):G_K \to \pi_1(\M_g)$.
As pointed out in \cite{MT}, $\rho \circ \sigma_{\M_g/K}(X)$ induces a map
$G_K \to \Out(\pi_1(\bar{X}))$ which is none other than the map obtained
from the exact sequence (\ref{fund}) by letting $\pi_1(X)$ act on
$\pi_1(\bar{X})$ by conjugation. Combining this with
Mochizuki's theorem \ref{moch} gives:

\begin{thm}
For any field $K$ contained in a finite extension of a $p$-adic field, the section map $\sigma_{\M_g/K}$ is injective.
\end{thm}

The following result confirms a conjecture of
Stoll \cite{Stoll} if we assume that $\sigma_{\M_g/K}$ surjects onto
$S_0(K,\M_g)$.

\begin{thm}
Assume that $\sigma_{\M_g/K}(\M_g(K)) = S_0(K,\M_g)$ for all $g>1$ and all
number fields $K$. Then $\sigma_{X/K}(X(K)) = S(K,X)$ for all smooth
projective curves of genus at least two and all number fields $K$.
\end{thm}

\begin{proof}
For any algebraic curve $X/K$ there is a non-constant map $X \to \M_g$
with image $Y$, say, for some $g$,
defined over an extension $L$ of $K$, given by
the Kodaira-Parshin construction. This gives a map
$\gamma: \pi_1(X\otimes L) \to \pi_1(\M_g \otimes L)$, over $L$. Let $s \in S(K,X)$,
then $\gamma \circ (s |_{G_L}) \in S_0(L,\M_g)$ and the assumption of the theorem yields that
$\gamma \circ (s |_{G_L}) = \sigma_{\M_g/L}(P), P \in \M_g(L)$. We can combine this with
the injectivity of $\sigma_{\M_g/K_v}$ (Mochizuki's theorem) to deduce that
in fact $P \in Y(L_v)\cap \M_g(L) = Y(L)$. We can consider the pullback to
$X$ of the Galois orbit of $P$, which gives us a zero dimensional scheme in
$X$ having points locally everywhere and, moreover, being unobstructed by
every abelian cover coming from an abelian cover of $X$. By the work of
Stoll \cite{Stoll}, Proposition 5.2, we conclude that $X$ has a rational point corresponding
to $s$.
\end{proof}

\begin{merci} The second author would like to thank J. Achter, D. Harari, E. Ozman, T. Schlank,
and J. Starr for comments and information. He would also like
to thank the Simons Foundation (grant \#234591) and the
Centre Bernoulli at EPFL for financial support.

The third author (Y.Z.) is grateful to Frans Oort, Ching-Li Chai and Jiangwei Xue  for helpful
discussions and to the Simons Foundation  for financial and moral
support (via grant \#246625 to Yuri Zarkhin). Part of this work was done in May--June 2015 when he was  visiting Department of Mathematics of the Weizmann Institute of Science (Rehovot, Israel).  The final version of this paper was prepared in May-June 2016 when he was a visitor at the Max-Planck-Institut f\"ur Mathematik (Bonn, Germany). The hospitality and support of both Institutes are gratefully acknowledged.

We are very grateful to the anonymous referees, whose careful readings and comments have greatly improved the readability of this paper. We would also like to thank W. Sawin for comments.

\end{merci}

Department of Mathematics, University of Utah

Salt Lake City, UT 84103, USA

e-mail\textup{: \texttt{patrikis@math.utah.edu}}

Department of Mathematics, University of Texas

Austin, TX 78712, USA and

School of Mathematics and Statistics, University of Canterbury,

Private Bag 4800, Christchurch 8140, New Zealand

e-mail\textup{: \texttt{voloch@math.utexas.edu}}

Department of Mathematics, Pennsylvania State University,

University Park, PA 16802, USA

e-mail\textup{: \texttt{zarhin@math.psu.edu}}

\end{document}